\documentclass[12pt]{amsart}

\usepackage{nag}

\usepackage{amssymb}
\usepackage{mathtools,arydshln}
\usepackage[utf8]{inputenc}

\usepackage{mathrsfs}
\usepackage[shortlabels]{enumitem}

\usepackage[dvipsnames]{xcolor}

\usepackage[pagebackref,colorlinks,linkcolor=Maroon,citecolor=MidnightBlue,urlcolor=NavyBlue,hypertexnames=true]{hyperref}

\usepackage{ulem,xspace,titletoc,microtype}

\usepackage{bbm}

\renewcommand{\qedsymbol}{\rule[.12ex]{1.2ex}{1.2ex}}
\renewcommand{\subset}{\subseteq}
\allowdisplaybreaks

\makeatletter
\def\moverlay{\mathpalette\mov@rlay}
\def\mov@rlay#1#2{\leavevmode\vtop{
		\baselineskip\z@skip \lineskiplimit-\maxdimen
		\ialign{\hfil$#1##$\hfil\cr#2\crcr}}}
\makeatother

\newcommand{\plangle}{\moverlay{(\cr<}}
\newcommand{\prangle}{\moverlay{)\cr>}}
\newcommand{\ff}{\C\plangle x \prangle}
\newcommand{\ffa}{\C\plangle a \prangle}
\newcommand{\fftwo}{\C\plangle a,x \prangle}
\newcommand{\ffthree}{\C\plangle h,a,x\prangle}

\makeatletter
\newsavebox\myboxA
\newsavebox\myboxB
\newlength\mylenA

\newcommand*\xoverline[2][0.75]{%
    \sbox{\myboxA}{$\m@th#2$}%
    \setbox\myboxB\null% Phantom box
    \ht\myboxB=\ht\myboxA%
    \dp\myboxB=\dp\myboxA%
    \wd\myboxB=#1\wd\myboxA% Scale phantom
    \sbox\myboxB{$\m@th\overline{\copy\myboxB}$}%  Overlined phantom
    \setlength\mylenA{\the\wd\myboxA}%   calc width diff
    \addtolength\mylenA{-\the\wd\myboxB}%
    \ifdim\wd\myboxB<\wd\myboxA%
       \rlap{\hskip 0.5\mylenA\usebox\myboxB}{\usebox\myboxA}%
    \else
        \hskip -0.5\mylenA\rlap{\usebox\myboxA}{\hskip 0.5\mylenA\usebox\myboxB}%
    \fi}
\makeatother

\newtheorem{theorem}            {Theorem}[section]
\newtheorem{corollary}          [theorem]{Corollary}
\newtheorem{proposition}        [theorem]{Proposition}

\newtheorem{lemma}              [theorem]{Lemma}
\newtheorem{remark}         [theorem]{Remark}

\newtheorem{example}            [theorem]{Example}

\makeindex
\makeatletter
\newcommand{\mycontentsbox}{%
{\parskip=1.0pt
\newpage\centerline{NOT FOR PUBLICATION}\tableofcontents\printindex}}
\def\enddoc@text{\ifx\@empty\@translators \else\@settranslators\fi
\ifx\@empty\addresses \else\@setaddresses\fi
\newpage\mycontentsbox%\newpage\printindex
}
\makeatother

\contentsmargin{2.55em} 

 \numberwithin{equation}{section}

\newcommand{\gv}{{\tt{g}}}
\newcommand{\hv}{{\tt{h}}}

\newcommand{\tg}{{\tt{h}}}

\def\cD{ {\mathcal D} }
\def\cE{ {\mathcal E} }

\def\cH{ {\mathcal H} }

\def\cS{{\mathcal S} }

\def\cV{{\mathcal V}}

\def\cU{\mathcal U}

\newcommand{\dom}{\operatorname{dom}}
\newcommand{\domp}{\dom^{+}}

\newcommand{\domkeb}{\operatorname{dom}_{\rm ver}}
\newcommand{\domkebp}{\domkeb^+}

\def\bbS{{\mathbb S}}

\def\C{\mathbb C}

\def\R{\mathbb R}

\def\range{\operatorname{rng}}
\def\ran{\operatorname{rng}}

\newcommand{\SMR}{SMR\xspace}
\newcommand{\kebab}{vertebral\xspace}
\newcommand{\robust}{full\xspace}

\newcommand{\dis}{disaggregate\xspace}

\mathchardef\mhyphen="2D

\def\cH{\mathcal{H}}

\def\tY{\widetilde{Y}}
\def\ax{\langle x \rangle}
\def\fatwo{\C\langle a,x\rangle}
\def\faa{\C\langle a\rangle}

\newcommand{\axy}{\langle x,y \rangle}
\def\Span{\operatorname{span}}
\def\tA{\widetilde{A}}

\newcommand{\tz}{\widetilde{z}}

\newcommand{\Hxy}{H^{{\scriptstyle xy}}}
\newcommand{\msS}{\mathscr{S}}

\newcommand{\etwo}{e}
\newcommand{\ethree}{\sx}

\newcommand{\Ep}{e}

\newcommand{\sE}{\mathscr{E}}

\newcommand{\listL}{\mathscr{L}}
\newcommand{\listLp}{{\listL_*}}

\newcommand{\IV}{(I_\mu\otimes V)}

\newcommand{\pp}{\rho}

\newcommand{\afflin}{\lambda}

\newcommand{\sx}{x}
\newcommand{\as}{\langle \sx_1,\sx_2 \rangle}

\newcommand{\PL}{[\listL]}

\newcommand{\Bxy}{\mathfrak{B}}
\newcommand{\Mxy}{\mathfrak{M}}

\newcommand{\hatT}{\widehat{T}}

\newcommand{\SB}{S}
\newcommand{\TB}{T}
\newcommand{\JB}{J}

\newcommand{\jpf}{f\!\!\!f}
\newcommand{\VT}{V_T}

\newcommand{\domdag}{\dom^{\ddag}}

\newcommand{\bS}{\mathbb{S}}
\newcommand{\vg}{{\tt{g}}}
\newcommand{\vh}{{\tt{h}}}

\newcommand{\hatdelta}{\widehat{\delta}}
\newcommand{\hatbeta}{\widehat{\beta}}

\newcommand{\vk}{{\tt{k}}}

\newcommand{\tta}{{\tt{a}}}
\newcommand{\ttx}{{\tt{x}}}

\newcommand{\col}{\operatorname{col}}

\newcommand{\XX}{{\tt{X}}}
\newcommand{\Omegap}{\Omega^+}

\newcommand{\df}[1]{{\bf{#1}}{\index{#1}}}

\title[Noncommutative partially convex rational functions]{Noncommutative partially convex rational functions}

\author[M.T. Jury]{Michael Jury${}^1$}
\address{Michael Jury, Department of Mathematics\\
  University of Florida\\ Gainesville }
\email{mjury@ufl.edu}
\thanks{${}^1$Research Supported by NSF grant DMS-1900364.}

\author[I. Klep]{Igor Klep${}^{2}$}
\address{Igor Klep, Faculty of Mathematics and Physics, 
University of Ljubljana, Slovenia}
\email{igor.klep@fmf.uni-lj.si}
\thanks{${}^2$Supported by the 
Slovenian Research Agency grants J1-2453, J1-8132 and P1-0222. Partially supported by the 
Marsden Fund Council of the Royal Society of New Zealand.}

\author[M.E. Mancuso]{Mark E. Mancuso${}^3$}
\address{Mark E. Mancuso, Department of Mathematics and Statistics, 
  Washington University in St. Louis%\\ St. Louis
  }
\email{mark.mancuso@wustl.edu}
\thanks{${}^3$Research partially supported by NSF grant DMS-1565243}

\author[S. McCullough]{Scott McCullough${}^4$}
\address{Scott McCullough, Department of Mathematics\\
  University of Florida\\ Gainesville %\\
   }
   \email{sam@math.ufl.edu}
\thanks{${}^4$Research supported by  NSF grants DMS-1361501 and DMS-1764231}

\author[J.E. Pascoe]{James Eldred Pascoe${}^5$}
\address{James E. Pascoe, Department of Mathematics\\
  University of Florida\\ Gainesville %\\
   }
   \email{pascoej@ufl.edu}
\thanks{${}^5$Partially supported by NSF MSPRF DMS 1606260.}



\numberwithin{equation}{section}

\makeindex

\begin{document}
\fontsize{11.5pt}{13.5pt}\selectfont
\begin{abstract}
Motivated by classical notions of bilinear matrix inequalities (BMIs) and partial convexity, 
this article investigates partial convexity for noncommutative functions.
It is shown that noncommutative rational functions that are partially convex
 admit novel butterfly-type realizations that necessitate square roots. 
A strengthening of partial convexity arising in connection
 with BMIs -- $xy$-convexity --  is also considered.
 A characterization of $xy$-convex polynomials is given.
\end{abstract}

%--- arxiv plain text abstract -----------------
%Motivated by classical notions of bilinear matrix inequalities (BMIs) and partial convexity, this article investigates partial convexity for noncommutative functions. It is shown that noncommutative rational functions that are partially convex admit novel butterfly-type realizations that necessitate square roots. The notion  of xy-convexity, a strengthening of partial convexity arising in connection with BMIs, is also considered. A characterization of xy-convex polynomials is given.

\maketitle

\section{Introduction}

Convexity and its matricial analogs arise naturally in many mathematical and engineering 
contexts. A function $f:[a,b]\to\R$ is convex if 
\[
f\left(\frac{x+y}2\right) \leq \frac12\left(f(x)+f(y)\right)
\]
for all $x,y\in[a,b]$. Convex functions have good 
optimization properties. For example, local minima are global, making them highly desirable in applications.
The dimension-free or scalable matrix analog of convexity appears in many modern applications, such as 
linear systems engineering \cite{boyd,SIG}, 
wireless communication \cite{JB07}, matrix means \cite{And89,And94,Han81}, 
perspective functions \cite{Eff09,ENE11},
random matrices and free probability \cite{GS09}
and noncommutative function theory \cite{DK+,HMV06,HM04,DHM,BM}.
Often in systems engineering \cite{ohmp} 
problems have two classes of variables: known unknowns $a=(a_1,\ldots,a_\tg)$ and
unknown unknowns  $x=(x_1,\ldots,x_\vg)$. Linear system problems specified 
by a signal flow diagram naturally give rise to matrix inequalities $p(a,x)\succeq0$,
 where $p$ is a  polynomial, or more generally a rational function, 
 in freely noncommuting variables.  
The $a$ variables represent system parameters whose size, which can be large, depends
 upon the specific problem.  The $x$ variables represent the design variables.
 A key point is that  $p(a,x)$ depends only upon the
 signal flow diagram.  Thus a choice of a value $A$ for $a$  corresponds to 
 a specific problem  governed by the given signal flow diagram and
 in that sense $a$ is a known unknown.  One then chooses
 the design variable $X$ to optimize an  objective and 
 in that sense $x$ is an unknown unknown. 
  Partial convexity in the unknown unknowns
 $x$ is then sufficient for reliable numerics and optimization.

A function $f:(-1,1)\to\R$
is matrix convex if 
\[
f\left(\frac{X+Y}2\right) \preceq \frac12\left(f(X)+f(Y)\right)
\]
for all hermitian matrices $X,Y$ with spectrum in $(-1,1)$.
Matrix convex functions are automatically real analytic and 
admit analytic realizations, such as the famous Kraus formula \cite{Kra36,Bha97}
\begin{equation}\label{eq:kraus}
f(x)=a+bx+\int_{-1}^1 \frac{x^2}{1+tx}\,d\mu,
\end{equation}
where $a,b\in\R$ and $\mu$ is a finite Borel measure on $[-1,1]$. 
Conversely, functions of the form  \eqref{eq:kraus}
 are readily seen to be matrix convex on $(-1,1).$ As an  example, 
 the Kraus formula \eqref{eq:kraus} in
 conjunction with the asymptotics at infinity  shows that
$x^2$ is matrix convex, but $x^4$ is not. 

In the noncommutative multivariable setting one considers noncommutative (nc) polynomials, 
rational functions and their generalizations. An nc polynomial is a linear combination of
 words in the freely noncommuting  letters $x=(x_1,\ldots,x_\gv)$.  For example,
\begin{equation}\label{eq:polyP}
p(x)=x_1x_2-17x_2x_1+4
\end{equation}
 is a nc (or free) polynomial. 
 Noncommutative polynomials are  naturally evaluated at tuples of matrices of any size. 
 For instance, to evaluate  $p(x)$ from \eqref{eq:polyP}
 on
 \[
X_1=\begin{pmatrix}
1 & 2 \\ 3 & 4
\end{pmatrix},
\qquad
X_2=\begin{pmatrix}
-1 & -1 \\ -1 & -1
\end{pmatrix},
 \]
 we substitute $X_i$ for the variable $x_i$, that is,
\[
p(X_1,X_2)=X_1X_2-17X_2X_1+4I_2=\begin{pmatrix}
 69 & 99 \\
 61 & 99 \\
\end{pmatrix}.
\]
More generally, an nc rational function is a
syntactically valid expression involving $x, +, \cdot, ( )^{-1}$ and scalars. 
Thus
\[
r(x)=1+(x_1-x_2(x_1x_2-x_2x_1)^{-1})^{-1}
\]
is an example of a nc rational function. It is evaluated
at a tuple $X=(X_1,X_2)$ of $n\times n$ matrices
 for which $X_1X_2-X_2X_1$ is invertible
 and in turn $X_1-X_2(X_1X_2-X_2X_1)^{-1}$ is invertible in the
 natural way to output an $n\times n$ matrix $r(X).$
A  nc rational function $r$ is \df{symmetric} if $r(X)=r(X)^*$ for all
hermitian tuples $X$ in its domain.

Matrix convexity for multivariate nc functions is now well understood. 
Analogs of the Kraus representation,
the so-called butterfly realizations,
were obtained in \cite{HMV06} for rational functions
and in \cite{PTD} for more general nc functions.
There is a paucity of matrix convex polynomials: 
as first observed in \cite{HM04}  they are of degree at most two. 

A main result of this paper, Theorem~\ref{t:wurzelificationINTRO},
is  an analog of the Kraus representation for 
partially convex
nc rational functions. 
Specialized to polynomials, our results extend and generalize results of \cite{HHLM}.
Moreover, we also investigate 
the stronger notion of  $xy$-convexity, modeled on the theory of bilinear matrix inequalities (BMIs) \cite{KSVS04}.

\subsection{Main results}
For positive integers $\vk$ and $n$, let \df{$\bbS_n^\vk = \bbS_n^\vk(\C)$} denote the $\vk$-tuples
of $n\times n$ hermitian matrices over $\C.$  A
 subset $\cD=(\cD_n)_n$ of $\bS^{\vk}$ is 
 a sequence of sets such that $\cD_n \subseteq \bS_n^{\vk}.$
 This subset is 
\df{free}, or a \df{free set},  if it is \df{closed under direct sums}
 and {\bf unitary conjugation}: \index{closed under unitary conjugation}
 if $Y\in\cD_m,$ $X\in\cD_n,$ and $U$ is an $n\times n$ unitary matrix, then 
\[
\begin{split}
 X\oplus Y & := \begin{pmatrix} X_1\oplus Y_1,  & \cdots, & X_{\vk}\oplus Y_{\vk}\end{pmatrix}
  = \left ( \begin{pmatrix} X_1&0\\0&Y_1\end{pmatrix},  \,  \cdots,  \, 
       \begin{pmatrix} X_{\vk}&0\\0&Y_{\vk}\end{pmatrix} \right )\in\cD_{n+m},
\\
U^*XU & := (U^*X_1U,\ldots,U^*X_{\vk}U)\in\cD_n.
\end{split}
\]
It is open if each $\cD_n$ is open. (In general adjectives 
 such as open and connected apply term-wise to $\cD$.)

Since we are dividing our freely noncommuting variables into 
 two classes $a=(a_1,\dots,a_\tg)$ and $x=(x_1,\dots,x_{\vg}),$
 where $\vg$ and $\tg$ are positive integers, we take
 $\vk=\tg + \gv$ and 
 let $\bS^{\vk}=\bS^{\tg}\times \bS^{\gv}=(\bS_n^{\tg}\times\bS_n^{\gv})_n.$
 We express elements of $\bS^{\vk}_n$ as $(A,X),$ where
 $A\in \bS^{\tg}$ and $X\in \bS^{\gv}.$

The symmetric version \cite[Proposition 4.3]{HMV06} of the well-known 
Schützenberger \cite{Scu61}
state space similarity theorem  implies that a symmetric nc rational function $r(a,x)$ 
 that is \df{regular}  at the origin (has $0$ in
 its domain)  admits a 
 symmetric realization
\begin{equation}\label{eq:ratsSMRintro}
r(a,x)=c^*\big(J -\sum_{i=1}^\gv T_i x_i-\sum_{j=1}^\tg S_ja_j\big)^{-1}c,
\end{equation}
where, for some positive integer $e$,  the $e\times e$ matrix 
  $J$ is a signature matrix ($J^2=I,\, J^*=J$), the $e\times e$
 matrices  $S_j,T_i$ are hermitian and $c\in \C^e.$
 In the case $e$ is the smallest such positive integer,  the \index{SMR}
 resulting realization is a \df{symmetric minimal realization (SMR) of size $e$}.
 Any two SMRs that determine the same rational function are
 similar as explained in more detail
 in Subsection \ref{sec:prelims}. In particular, the definitions and results 
  here  stated in terms of an SMR do not depend upon the  choice
 of SMR.  The results of \cite{volcic,KVVdomain} justify defining 
 the \df{domain} of $r$ as  \index{$\dom r$}
\begin{equation}
\label{def:domr}
 \dom r = \{(A,X)\in \bbS^{\tg}\times\bbS^\vg: 
   \det \Big(J\otimes I-\sum_{i=1}^\gv T_i \otimes X_i-\sum_{j=1}^\tg S_j\otimes 
   A_j\Big) \ne 0\}.
\end{equation}
 In particular, the domain of a rational function is a free open set. 
Let $\fftwo$ \index{$\fftwo$}
 denote the set of rational functions in the variables $a$ and $x.$

\subsubsection{The domain of partial convexity}
An nc rational function $r$
 is  \df{matrix convex in $x$} or \df{partially convex} on $\cD$ if
 \[
r\left(A,\frac{X+Y}2\right) \preceq \frac12\left(r(A,X)+r(A,Y)\right)
 \]
 whenever $(A,X),(A,Y),(A,\frac{X+Y}{2})\in\cD$. Sublevel  
 sets of such functions have matrix convexity
 properties, which we do not discuss here save to  note that 
 these sublevel sets are very important in real 
and convex algebraic geometry, polynomial optimization, and the rapidly emerging 
subject of noncommutative function theory \cite{SSS,Pop,PSS,PS,KVV,HMannals,HL,HKMjems,HKM13,EE+,EE18,DDSS,BMV}.

 Our first main theorem gives an effective easily computable 
 criterion to determine  where $r$ is convex in $x.$ 
 To state this result,  let \df{$\VT$}   denote the inclusion
 of the span of the ranges of the $T_j$ into $\C^e$ and let \index{$R_T$}
\begin{equation}
 \label{def:RsubT}
 R_T(a,x) = \VT^* \, \left (J-\sum_{i=1}^\gv T_i x_i
    -\sum_{j=1}^\tg S_ja_j \right)^{-1}\, \VT.
\end{equation}
Finally, let  \index{$\domp r$}
\begin{equation}
 \label{def:domplus}
   \domp r:=\{(A,X)\in \dom r: R_T(A,X)\succeq 0\}.
\end{equation}

 Given $\cD\subset \bbS^{\hv} \times \bbS^{\gv}$ and $A\in \bS_k^{\tg},$  
\begin{equation}
 \label{e:sliceA}
   \cD[A]=\{X\in \bS_k^{\gv}: (A,X)\in\cD\}.
\end{equation}
A free set $\cD$ is convex (resp. open) in $x$ if  \df{$\cD[A]$} is convex (resp. open)
 for each $A\in \bS^{\tg}.$  \index{convex in $x$} \index{open in $x$}
 Theorem~\ref{t:characterINTRO} below, which is proved as Theorem~\ref{t:character}, says that $\domp r$ 
 deserves the moniker, the {\it  domain of partial convexity of $r.$}
 Generally, a free set $\cD$ is a \df{domain of partial convexity}
 for $r$ if $\cD$ is open in $x,$ convex in $x,$ and $r$
 is convex in $x$ on $\cD.$ It is a \df{\robust domain of partial convexity}
 if in addition $\cD$ contains
 a free open set $\cU$ with $\cU_1\ne \emptyset.$

\begin{theorem}\label{t:characterINTRO}
 The set $\domp r$ is a domain of partial convexity for $r.$ 

 Conversely, if 
 $\cD\subset \dom r$ is a \robust domain of partial convexity for $r,$
  then $\cD\subset\domp r$ and $\domp r$ is also a \robust domain
 of partial convexity for $r.$
\end{theorem}

\subsubsection{The root butterfly realization: a certificate of partial convexity}
Our second  main theorem,  the {\it root butterfly realization,}
gives an algebraic certificate for partial convexity near 
points in the domain of $r$ of the form $(A,0).$ 
This realization differs from  existing realizations in that
it contains a square root that  appears difficult to avoid.
A free set $\cD$ is a \df{\kebab}
 set if $(A,X)\in \cD$ implies  $(A,0)\in\cD.$ 
 We denote the positive (semidefinite) square root of
 a positive (semidefinite) matrix $P$ by \df{$\sqrt{P}.$} 
 A \kebab free set $\cD$ is a \df{\kebab domain of convexity for $r$}
 provided $\cD$ is open in $x,$ convex in $x,$ 
 and if $r$ is convex in $x$ on $\cD.$ If in addition
 $\cD$ contains a free open set $\cU$ with $\cU_1\ne \emptyset,$
 then $\cD$ is a \df{\robust \kebab domain of convexity}.

 The \df{\kebab domain} of $r$ is the set \index{$\domkeb$}
\[
 \domkeb r = \{(A,X) \in \dom r: (A,0)\in\dom r\}.
\]
 Let \index{$\domkebp$}
\[
 \domkebp r =\{(A,X)\in\domp r: (A,0)\in\domp r\}
\]
 Theorem~\ref{t:wurzelificationINTRO}  gives a realization tailored
 to partial convexity that provides an algebraic
 certificate of convexity in $x$ for an $r\in \fftwo.$
Given a subset $\cD\subset \bbS^{\tg}\times\bbS^{\gv},$
let  \index{$\pi_a$}
\begin{equation}
\label{e:pisuba}
\pi_a(\cD)=\{A\in\bbS^{\tg}: (A,X) \in \cD \mbox{ for some } X\in\bbS^{\gv}\}.
\end{equation}

\begin{theorem}[Wurzelschmetterlingrealisierung]\label{t:wurzelificationINTRO}
Suppose $r\in\fftwo$ is a nc rational function 
with the \SMR as in \eqref{eq:ratsSMRintro}. Then
\begin{enumerate}[\rm (1)]
 \item 
   $\domkebp r$ is a \kebab domain of convexity for $r;$
\item  if $\cD$ 
 is a \robust \kebab domain of convexity for $r,$  then $\cD\subset \domkebp r$
 and $\domkebp r$ is a also a \robust \kebab domain of convexity for $r;$
\item \label{i:>=0} there exists  a positive integer $k$, a tuple
  $\hatT\in M_k(\C)^\vg,$ and 
a symmetric rational function $w(a)\in \ffa^{k\times k}$   
    defined on 
$\pi_a(\domkeb r)$    such that
\[
  \domkebp r =  \Big\{(A,X)\in \domkeb r: w(A)\succeq 0 , \ \ 
 I-  \sqrt{w(A)}\, \left [\sum_{i=1}^{\vg}  \hatT_i  \otimes X_i\right ]
  \sqrt{w(A)} \succ 0\Big\};
\]
 \item   there exists  a rational function $\ell(a,x)\in \fftwo^{k\times 1},$
   defined on $\domkeb r$ and  linear in $x;$   and
 a symmetric rational function  $\jpf(a,x)\in \fftwo,$ defined on 
    $\domkeb r$ and  affine linear in $x$
 such that  $r$ 
 admits the following realization, valid on $\domkebp r:$
\begin{equation*} %\label{eq:wurschtINTRO}
r=\ell(a,x)^*\sqrt {w(a)} \, \left 
  (I-\sum \sqrt {w(a)}\hatT_ix_i\sqrt {w(a)}\right )^{-1} \,
    \sqrt {w(a)}\ \ell(a,x) + f\!\!\!f (a,x).
\end{equation*}
\end{enumerate}
\end{theorem}

As a corollary
 we  obtain the following simple representation for polynomials that are convex in $x$.
We use \index{$\fatwo$} to denote the set of noncommutative polynomials in $(a,x)$.

\begin{corollary}[{\cite[Proposition~3.1]{HHLM}}]
\label{c:xcvxINTRO}
 Suppose $\cD$ is a  free set that is open in $x,$ convex in $x$ 
 and contains a free open set $\cU$ such that $\cU_1\ne\emptyset.$ 
A polynomial $p(a,x)$ is convex in $x$ on $\cD$ 
if and only if there exists $\ell(a,x)\in\fatwo$ that
is linear in $x$, and a symmetric $w(a)\in\faa$ 
that is positive semidefinite  on $\pi_a(\cD)$ such that 
\begin{equation*} %\label{e:polyxcvxINTRO}
p=\ell(a,x)^* w(a)\ell(a,x) + f\!\!\!f (a,x),
\end{equation*}
where $f\!\!\!f (a,x)\in\fatwo$ is affine linear in $x$  and symmetric.
In particular, if $p$ is convex in $x$ on $\cD,$ then
 $p$ is convex in $x$ on $\pi_a(\cD)\times\bbS^{\gv}.$
\end{corollary}

\subsubsection{$xy$-convexity and BMIs}
 In this subsection we preview our results on $xy$-convexity and BMIs.
 Like partial convexity, here we have two classes of variables.
 Unlike partial convexity, the roles of the classes of variables
 appear symmetrically in $xy$-convexity. With that in mind, 
 we switch notation somewhat and consider 
 freely noncommuting letters $x_1,\dots,x_{\vg},y_1,\dots,y_{\vh}.$
 
 An expression of the form
\begin{equation*}
%\label{e:BMI}
L(x,y) = A_0 + \sum_{j=1}^{\vg} A_j x_j + \sum_{k=1}^{\tg} B_k y_k 
  + \sum_{p,q=1}^{\vg,\tg} C_{pq} x_p y_q +    \sum_{p,q=1}^{\vg,\tg} D_{pq} y_q x_p,
\end{equation*}
 where $A_j,B_k,C_{pq},D_{pq}$ are all matrices of the same size, is an
 \df{$xy$-pencil}. In the case $A_j,B_k$ are hermitian and $D_{pq}=C_{qp}^*,$
 $L$ is a \df{hermitian $xy$-pencil}. If $A_0=I$, then $L$ is \df{monic}.
 For a monic hermitian $xy$-pencil $L$, the inequality $L(X,Y)\succeq 0$
 for $(X,Y) \in \bS^{\vg} \times \bS^{\vh}$ is a 
 \df{bilinear matrix inequality} (BMI) \cite{VAB,GLS96,KSVS04}. 
 Domains $\cD$ defined by BMIs are convex in the $x$ and $y$ variables separately. 

We say a function $f$ of two freely noncommuting variables
 is {\bf $xy$-convex} on a free set $\cD$ if
$f(V^*(X,Y)V)\preceq V^*f(X,Y)V$ for all isometries $V$,
and all $X,Y\in\cD$ satisfying
$V^*(XY)V=(V^*XV)(V^*YV)$. Such a pair $((X,Y),V)$ 
is called an \df{$xy$-pair}. Sublevel sets of $xy$-convex functions
are delineated by (perhaps infinitely many) BMIs as proved in \cite{JKMMP+}.

 Symmetric polynomials in two freely noncommuting variables $x$ and $y$ 
 (so $\vg=1=\vh$)
  that are $xy$-convex essentially arise from BMIs. Here $xy$-convex
 means globally; that is, on all of $\bbS^1 \times \bbS^1.$

\begin{theorem}
\label{t:introxyconvexp}
 Suppose $p$ is a symmetric polynomial in the two freely noncommuting variables 
 $x,y.$ If $p$ is  $xy$-convex, then there exists 
 a hermitian $xy$-pencil $\afflin \in \C\langle x,y\rangle$,
 a positive integer $k$  and an $xy$-pencil 
 $\Lambda\in \C\langle x,y\rangle^{k\times 1}$ such that 
\[
 p = \afflin(x,y) +\Lambda(x,y)^* \Lambda(x,y).
\]
 The converse is easily seen to be true.
\end{theorem}

 The notions of partial convexity and $xy$-convexity are two
 instantiations of $\Gamma$-convexity \cite{JKMMP+}. 
 Let $\cD \subset \bS^{\tg}\times \bS^{\gv}$ be a given
 free open set that is also closed with respect to 
  restrictions to reducing subspaces; that is 
  if $(A,X)\in \cD$ and $V$ is an isometry whose range
  reduces each $A_j$ and $X_k$, then $V^*(A,X)V\in \cD.$
  The set $\cD$ is \df{convex in $x$}, or \df{partially convex}, if 
  for each $A\in \bS_k^{\tg}$ the slice $\cD[A]$
 (see \eqref{e:sliceA}) is convex.
  Likewise $\cD$ is \df{$a^2$-convex} if for each $(A,X)\in \cD_n$
  and isometry $V:\C^m\to \C^n$ such that $V^*A^2V=(V^*AV)^2$
  it follows that $V^*(A,X)V\in \cD$. 
  In \cite{JKMMP+} it is shown that $\cD$ is convex in $x$
  if and only if it is $a^2$-convex. 
  A  straightforward
  variation on the proof of that result establishes
  Proposition \ref{p:a2convex} below.
  A rational function $r\in \fftwo$ is \df{$a^2$-convex}
  on $\cD$ if, whenever $(A,X)\in \cD$ and $V:\C^m\to \C^n$ 
 is an isometry such that $V^*A_j^2V=(V^*A_jV)^2$
 and $V^*(A,X)V\in \cD$, we have that
\[
 V^*r(A,X)V \succeq r(V^*(A,X)V).
\]

\begin{proposition}
 \label{p:a2convex}
 If $\cD\subset \bS^{\tg}\times \bS^{\gv}$ is a free set
 that is closed with respect to reducing subspaces and 
  $a^2$-convex, then an $r\in \fftwo$  is $a^2$-convex on $\cD$ 
  if and only if it is  convex in $x$   on $\cD.$
\end{proposition}

 A proof of Proposition~\ref{p:a2convex} appears in
 appendix~\ref{s:a2convex}.

\section{Partial convexity for nc rational function}
\label{sec:partial}
 In this section we consider partial convexity 
of nc rational functions
 and establish
 Theorems \ref{t:characterINTRO}  and \ref{t:wurzelificationINTRO}
 as well as Corollary \ref{c:xcvxINTRO}. 

\subsection{Preliminaries}
\label{sec:prelims}
Proposition~\ref{p:state-space-similarity} below is a version of the well known
 state space similarity theorem
due to Sch\"utzenberger \cite{Scu61}; see also \cite{BMG05} or \cite[Proposition~4.3]{HMV06}.\looseness=-1

\begin{proposition}
\label{p:state-space-similarity}
 If
\[
 q(x) = a^*\left (J-\sum_{j=1}^m  A_j x_j\right)^{-1} a, \quad
 q(x) = b^* \left (K-\sum_{j=1}^m  B_j x_j \right)^{-1}b
\]
 are two SMRs for the same rational function, then 
 there is a unique matrix $S$ such that  $S^*KS=J,$
 $SJA_j= KB_jS$ for $1\le j\le m$ and $SJa=Kb.$ 
\end{proposition}

 A bit of algebra reveals that $S^*BS=A.$
 Thus  $K-\sum B_j x_j = S^*(J-\sum A_j x_j)S$ and it follows
 that the definitions of  $\dom r$, $\domp r$ and $\domdag r$
 are independent of the choice of SMR.

 Just as in the commutative case, it is well known that 
 convexity properties of a free rational
 functions can be characterized by positivity of a Hessian.
 See for instance \cite{HM98}.
 The \df{$x$-partial Hessian} of an SMR as in equation \eqref{eq:ratsSMRintro}
 is the rational function in $2\vg+\tg$ freely noncommuting variables, \index{$r_{xx}$}
\begin{equation}\label{e:rxx}
\begin{split}
r_{xx}(a,x)[h]&= 2 c^* R(a,x) (\sum_i T_i h_i) R(a,x) (\sum_i T_i h_i) R(a,x) c 
\\ &  = 2 \left [c^* R(a,x) (\sum_i T_i h_i)\right ] \,  R_T(a,x)\,  
 \left [(\sum_i T_i h_i) R(a,x) c \right],
\end{split}
\end{equation}
 where $R$ is the \df{resolvent} \index{$R(a,x)$}
\begin{equation}
\label{def:resolvent}
 R(a,x):= (J -\sum T_j x_j - \sum S_k a_k)^{-1},
\end{equation}
$\Lambda_T[h]=\sum_{j=1}^{\vg} T_j h_j$,
and $R_T(a,x)=\VT^* R(a,x)\VT$ is defined as in \eqref{def:RsubT}.
 Compare with 
\cite[Equation~(5.3)]{HMV06} where the {\it full} Hessian
 of a SMR is computed in detail. The $x$-partial Hessian 
is naturally \df{evaluated}
 at a tuple $(A,X,H)\in \bS^{\tg}\times \bS^{\vg}\times \bS^{\vg}$ where 
 $(A,X)\in \dom r$ 
 with  output a symmetric $k\times k$ matrix.

 Proposition \ref{p:HM98} is the partial convexity analog of
 the  \cite{HM98}
characterization of convexity in terms of Hessians. 
The proof is a straightforward modification of the one 
in \cite{HM98} so is only sketched below.

\begin{proposition}
\label{p:HM98}
 The rational function $r$ is convex in $x$ on a nonempty, open in $x,$ and convex in 
 $x$  set   $S\subset \dom r \cap \left (\bS_k^{\tg}\times \bS_k^{\vg} \right)$
 if and only if 
 $r_{xx}(A,X)[H]\succeq 0$ for all  $(A,X)\in S$ and  $H\in \bS_k^{\gv}.$
\end{proposition}

\begin{proof}[Sketch of proof]
 The rational function $r$ is convex in $x$ on $S$ if and only if for
 each $A\in \bS_k^{\tg}$ and 
 each positive linear functional $\lambda:\bS_k\to \mathbb R$
 the function $f_{A,\lambda}:S\to\mathbb R$ defined by
 $f_{A,\lambda}(X)=\lambda\circ r(A,X)$ is convex. On the other hand,
 $f_{A,\lambda}$ is convex if and only if its Hessian is positive; that
 is
\[
0\le  f^{\prime\prime}_{A,\lambda}(X)[H]
  = \lambda \circ r_{xx}(A,X)[H]
\]
for all $H$. Thus $f_{A,\lambda}$ is convex for 
 each $A$ and positive $\lambda$ if and only if $r_{xx}(A,X)[H] \succeq 0.$  
\end{proof}

\subsection{$\domp r$ is open in $x$ and convex in $x$}
\label{s:inx}
 In this section we show that $\domp r$ is both open in $x$ and convex in $x.$ 
 Let positive integers $m$ and $n,$ a 
 matrix $D\in\bbS_n$ and a matrix $B\in M_{m,n}(\C)$
 be given. Let $V:\C^m\to \C^m\oplus \C^n$ denote the inclusion,
\[
 Vx = \begin{pmatrix} x\\ 0 \end{pmatrix} 
 \in %
 \C^m\oplus\C^n.
\]
 Define $L:\bbS_m \to \bbS_{m+n}$ by 
\[
 L(\XX) = \begin{pmatrix} \XX & B\\ B^* & D\end{pmatrix}.
\] 
Let 
\[
 \Omega =\{\XX\in \bbS_n(\C): \det L(\XX) \ne 0\}
 \quad\text{and}\quad
 \Omegap = \{\XX\in \Omega: V^* L(\XX)^{-1} V \succeq 0\}.
\]

\begin{lemma}
\label{l:openinx}
The set $\Omegap$ is open, convex, and a 
  connected component of $\Omega.$
\end{lemma}

Before proving Lemma~\ref{l:openinx}, we first establish
the following result.

\begin{lemma}
 \label{l:Omega1}
   There exists a subspace $\cH\subseteq \C^m$ 
   and a self-adjoint operator $F$ on $\cH$
   such that, with $W$ equal the inclusion of $\cH$ into $\C^m,$ 
\begin{enumerate}[\rm (1)]
\item \label{i:Omi}
  $\XX\in\Omega$ if and only if $W^*\XX W-F$ is invertible; and 
 \item \label{i:Omii}
  $\XX\in\Omegap$ if and only if
$
  W^* \XX W - F \succ 0.
$
\end{enumerate}
\end{lemma}

\begin{proof}
The proof  is straightforward in the case that $D$
is invertible.  Indeed, under the assumption 
 that $D$ is invertible, a standard Schur complement
 result says $L(\XX)$ is invertible if and only if 
 the Schur complement of $D,$
\[
 S(\XX)= \XX - B D^{-1} B^*, 
\]
 is invertible and further, in that case, 
\[
 V^* L(\XX)^{-1} V = S(\XX)^{-1}.
\]
Thus the result holds with $\cH=\C^m$ and $F=BD^{-1}B^*.$

 The result also holds trivially if  $\Omega=\emptyset$
 by choosing $\cH=\{0\}.$ Thus, for the remainder
 of this proof, assume $D$ is not 
 invertible and $\Omega\ne \emptyset.$ In particular,
 $\ker D\cap \ker B \ne \{0\}.$

  With respect to the  orthogonal direct sum
 $\C^n = \ker D \oplus \ker D^\bot,$ 
\[
  D=\begin{pmatrix} 0 &0 \\0 & D_0\end{pmatrix}
\quad\text{and}\quad 
 L(\XX) = \begin{pmatrix} \XX & B_1 & B_2 \\
  B_1^* & 0 & 0 \\ B_2^* & 0 & D_0 \end{pmatrix},
\]
 with $D_0$ invertible. 
 It follows  that $B_1: \ker D\to \C^m$ is one-one,
 as otherwise $L(\XX)$ is never  invertible, violating 
 the assumption $\Omega\ne \emptyset.$ 
 
 With respect to the orthogonal decomposition
  $\C^m = \range B_1 \oplus \range B_1^\bot,$
\[
   B_1 =\begin{pmatrix} B_{1,1} \\ 0 \end{pmatrix}: \ker D\to \C^m.
\]
 In particular, $B_{1,1}$ is invertible. In these
 coordinates ($\C^m = \range B_1 \oplus \range B_1^\bot$
 and $\C^n =\ker D \oplus \ker D^\bot$),
\[
 L(\XX) = \begin{pmatrix} \XX_{1,1} & \XX_{1,2} & B_{1,1} & B_{1,2}\\
   \XX_{1,2}^* & \XX_{2,2} & 0 & B_{2,2} \\ 
   B_{1,1}^* & 0 & 0 & 0 \\ 
   B_{1,2}^* & B_{2,2}^* & 0 & D_0\end{pmatrix}.
\]
 Since $D_0$ is invertible, $L(\XX)$ is invertible 
 if and only if the Schur complement of $D_0,$
\[
 T(\XX) = \begin{pmatrix} \XX_{1,1} & \XX_{1,2} & B_{1,1} \\
   \XX_{1,2}^* & \XX_{2,2} & 0  \\ 
   B_{1,1}^* & 0 & 0 \end{pmatrix}
 - \begin{pmatrix} B_{1,2} \\ B_{2,2} \\ 0 \end{pmatrix}  
 \, D_0^{-1} \, \begin{pmatrix} B_{1,2}^* & B_{2,2}^* &0
 \end{pmatrix},
 \]
is invertible.    Writing $T(\XX)$ as
\[
  \begin{pmatrix} \XX_{1,1} -C_{1,1} & \XX_{1,2} - C_{1,2}  & B_{1,1}\\
   \XX_{1,2}^* -C_{1,2}^* & \XX_{2,2} -C_{2,2} & 0 \\ 
   B_{1,1}^* & 0 & 0  \end{pmatrix},
\]
 observe that  $T(\XX)$ is invertible if and only if 
 $\XX_{2,2}-C_{2,2}$ is invertible, 
  proving item~\ref{i:Omi} with $\cH=\range B_1^\bot$
 and $F=C_{2,2}.$  Moreover, 
\[
 T(\XX)^{-1} 
 =\begin{pmatrix} 0 & 0 & B_{1,1}^{-1}\\ 0 & (\XX_{2,2}-C_{2,2})^{-1}
 & * \\ B_{1,1}^{-*} & *&* \end{pmatrix}.
\]
 Since the upper $3\times 3$ block of $L(\XX)^{-1}$ is
 $T(\XX)^{-1},$ it follows that 
 \[
  V^* L(\XX)^{-1}  V = \begin{pmatrix} 0 & 0 \\0& 
  (\XX_{2,2}-C_{2,2})^{-1} \end{pmatrix}.
\] 
 Hence $\XX\in \Omega^+$ %
 if and only if $\XX_{2,2}-C_{2,2}\succ 0,$ proving
 item~\ref{i:Omii} again with $\cH=\range B_1^\bot$
 and $F=C_{2,2}.$
\end{proof}

\begin{proof}[Proof of Lemma~\ref{l:openinx}]
 Since, by Lemma~\ref{l:Omega1}, $\XX\in \Omega^+$
 if and only if  $W^* \XX W -F \succ 0,$ the
 set $\Omega^+$ is both open and convex.  Since 
 $\Omegap$ is convex, to prove
 $\Omegap$ is a connected component of $\Omega,$ it suffices
 to prove $\Omegap$ is closed in $\Omega.$  To this end,
 suppose $(\XX_n)_n$ is a sequence from $\Omegap$ that converges
 to $\XX\in\Omega.$ It follows from Lemma~\ref{l:Omega1}
  that $W^* \XX_n W - F\succ 0$
 for each $n$ and hence, after taking a limit, 
  $W^* \XX W- F\succeq 0.$
 On the other hand, $\XX\in\Omega$ implies  $W^*\XX W-F$ is invertible 
 by Lemma~\ref{l:Omega1}. 
 Hence $W^* \XX W -F \succ 0$ and therefore $\XX\in \Omegap$
 by yet another application of Lemma~\ref{l:Omega1}.
\end{proof}

\begin{proposition}
 \label{p:dompr-inx}
  Suppose $r\in\fftwo$ is a nc rational function 
with the \SMR as in \eqref{eq:ratsSMRintro} 
 and $A\in \bbS_n^{\vh}.$  The set 
\[
 \Omega[A]^+ =\{X\in\bbS_n^{\vg}: (A,X)\in \domp r\}
\]
is open, convex and a connected component of the set 
\[
 \Omega[A]= \{X\in \bbS_n^{\vg}: (A,X)\in \dom r\} 
    \subseteq \bbS_n^\vg.
\]
\end{proposition}

\begin{proof}
 Let $N$ denote the size of realization. 
   Thus $J\in M_N(\C).$ Without loss of generality,
 we assume that $\ran T \oplus \ran T^\bot$  decomposes
 $\C^N$ as $\C^a\oplus \C^b.$
Express $J,S,T$
 with respect to  this orthogonal decomposition as
\[
 \JB =   \begin{pmatrix} \JB_{1,1} & \JB_{1,2}\\ \JB_{1,2}^* & \JB_{2,2} \end{pmatrix}, \ \ \
 \SB_k =  \begin{pmatrix} \SB_{k,0} & \SB_{k,1} \\ \SB_{k,1}^*  & \SB_{k,2} \end{pmatrix}, \ \ \
 \TB_j =  \begin{pmatrix} \TB_{j,0} & 0 \\ 0 & 0 \end{pmatrix}.
\]

Let  $B=\JB_{1,2}\otimes I - \sum \SB_{k,1}\otimes A_k\in M_{a,b}(\C)\otimes\bbS_n$ and
 $D=\JB_{2,2}\otimes I - \sum_k \SB_{k,2}\otimes A_k\in \bbS_b \otimes\bbS_n
 \subseteq \bbS_{bn}$
and define $L:\bbS_{am}\to \bbS_{am+bn}$ by
\[
 L(\XX) = \begin{pmatrix} \XX & B\\ B^* & D
 \end{pmatrix} 
\]
 and let $V$ denote the inclusion of $\C^{a}\otimes \C^n$
 into $(\C^a\otimes \C^n)\oplus \C^b \otimes \C^n.$
Let  $\Omega =\{\XX\in \bbS_{am}: \det L(\XX)\ne0\}$ 
and let
\[
 \Omegap =\{\XX\in \Omega: V^* L(\XX)^{-1} V\succeq 0\}.
\]
 By Lemma~\ref{l:openinx}, $\Omegap$ is open, convex
 and a connected component of $\Omega.$ In
 particular, $\Omegap$ is closed in $\Omega.$

 Define $\Lambda: \bbS_n^{\vg}\to \bbS_{an}$ by
\[
 \Lambda(X) = (\JB_{1,1}\otimes I-\sum_k S_{k,0} \otimes A_k) - \sum_j T_{j,0} \otimes X_j.
\]
Observe $\Lambda$ is affine linear, 
$\Omega[A] = \Lambda^{-1}(\Omega)$ and
 $\Omega[A]^+ = \Lambda^{-1}(\Omegap).$
 Thus, since $\Lambda$ is continuous and
 $\Omegap$ is open, $\Omega[A]^+$ is open.
Likewise, since $\Lambda$ is affine linear
 and $\Omegap$ is convex, $\Omega[A]^+$ is convex
 and thus connected.
 Finally, since  $\Omega[A]^+$ connected, to show it is a component
 of $\Omega[A],$ it suffices to observe that it is closed
 since it is the inverse image under the continuous 
 map $\Lambda|_{\Omega[A]}:\Omega[A]\to \Omega$ 
 of the closed (in $\Omega$) set
 $\Omegap.$
\end{proof}

\subsection{Characterization of partial convexity}
 Throughout this section we fix  an  SMR
 \eqref{eq:ratsSMRintro} for $r$,
and let $R(a,x)$ denote the resolvent of equation \eqref{def:resolvent}.
Recall the definitions of $R_T$ and $\domp r$ of
 equations \eqref{def:RsubT} and \eqref{def:domplus}.\looseness=-1

\begin{theorem}\label{t:character}
If $r\in\fftwo$ is a nc rational function 
with the \SMR as in \eqref{eq:ratsSMRintro}, then 
\begin{enumerate}[\rm (1)]
 \item \label{i:open-and-convex} $\domp r$ is a domain of partial convexity
  for $r;$ 
 \item \label{i:maximal} if $\cD\subset \dom r$ is a \robust
  domain of partial convexity for $r,$ then $\cD\subset \domp r.$
\end{enumerate}
\end{theorem}

\begin{corollary}[\cite{HMV06}]
\label{c:prelim-poly}
  Suppose $r\in \ff.$ If $r$ is convex in a free open set containing $0,$ then
   $\dom_0 r$, the component of  $\dom r$ containing $0$, is
  convex and $r$ is convex on $\dom_0 r.$
\end{corollary}

 It is straightforward to verify that $\domp r$ is a free set.
  That $\domp r$ is open in $x$ and convex in $x$ was established
 in Proposition~\ref{p:dompr-inx}. Thus to prove $\domp r$ is
 a domain of partial convexity for $r,$ it remains to prove 
 that $r$ is convex in $x$ on $\domp r$, a statement that 
 follows from  Proposition \ref{l:character} below.  
 Item \ref{i:maximal} of  Theorem~\ref{t:character} is an immediate consequence of the
 converse portion of Proposition \ref{l:character}.

\begin{proposition}
 \label{l:character}
 Let $r$ denote the rational function  of \eqref{eq:ratsSMRintro}
 and suppose $\cE\subseteq \dom r$ is a free set 
 that is open in $x$ and convex in $x.$ 

 If $R_T\succeq 0$ on $\cE,$ then $r$ is convex
 in $x$ on $\cE.$ 
 Conversely, if $\cE$ contains a free open set $\cU$ with $\cU_1\ne \emptyset,$
 and if $r$ is convex in $x$ on $\cE,$ then  $R_T\succeq 0$ on $\cE.$
\end{proposition}

\subsubsection{The CHSY Lemma}
 In this subsection we establish a variant of the  CHSY Lemma \cite{CHSY} (see also \cite{BK,Vol18}) suitable
 for a proof of Proposition \ref{l:character}, starting with the of independent interest
 Lemma \ref{l:precheesy} below.

\begin{lemma}
 \label{l:precheesy}
   If $\xi_1,\dots,\xi_K \in \ff$ are linearly independent rational functions in $\gv$ variables,
   $m$ is a positive integer 
   and $\cU$ is a free open subset of $\bbS^\gv$ with $\cU_1\ne \emptyset,$ 
  then there exists a positive integer $M$,
   an $X\in \cU_M$ and a matrix $w\in M_{m,M}(\C)$ such that 
\[
 \{\begin{pmatrix} w\, \xi_1(X) v \\ \vdots \\ w\,  \xi_K(X) v \end{pmatrix} : v\in \C^M\}
   = \C^{K}\otimes \C^m =\C^{Km}.
\]
\end{lemma}

\begin{proof}
  Let $\Xi =  \col \left (\xi_1, \dots, \xi_K \right ) \in M_{K,1}(\ff).$
 Let $\cS$ denote the set of pairs $(z,Y)$, where, for some $n$, $Y\in \cU_n$
 and $z \in M_{m,n}(\C).$  Given $(z,Y)\in \cS_n$, let 
\[
 \cV_{(z,Y)} = \{ (I_K \otimes z) \Xi(Y)v: v\in \C^n\}\subset \C^K\otimes\C^m.
\]
 Given $A=(z,Y)$ and $\tA=(\tz,\tY)$ both in $\cS$, let 
\[
 A\oplus \tA  = \left(\begin{pmatrix} z & \tz \end{pmatrix},
   \begin{pmatrix} Y &0\\ 0 & \tY \end{pmatrix}
  \right).
\]
 It is straightforward to verify that $\cV_{A\oplus \tA} = \cV_A + \cV_{\tA}.$ Hence, 
 there exists a (dominating) pair $(w,X)\in \cS$ such that
\begin{equation}
\label{e:dominate}
  \cV_{(z,Y)} \subset \cV_{(w,X)},
\end{equation}
  for all $(z,Y)\in \cS.$ Suppose  $\alpha \in \cV_{(w,X)}^\perp.$ From equation \eqref{e:dominate},
 it follows that  $\alpha \in \cV_{(z,Y)}^\perp$ for all $(z,Y)\in \cS.$    Write $\alpha \in \C^K \otimes \C^m$
 as $\alpha =\sum \alpha_j \otimes e_j,$ where $\{e_1,\dots,e_m\}$ is the standard
 orthonormal  basis for $\C^m$ and $\alpha_j\in \C^K.$
   We will show, for each $j,$
 that $\sum_{s=1}^K \overline{(\alpha_j)_s}  \xi_s =0,$ and hence, by the linear independence assumption,
 that each $\alpha_j$, and hence $\alpha,$ is zero. 
 Accordingly, fix $j$ and let $n$ and $Y\in \cU_n$ be given. Given a vector $f\in \C^n,$
 let $w_f = e_j f^*.$ Since $\alpha \in \cV_{(Y,w_f)}^\perp$,
\[
 0 = \alpha^* [ I_K \otimes w_f] \Xi(Y) = (\alpha_j^*\otimes f^*) \Xi(Y)
     = f^* \sum_{s=1}^K \overline{(\alpha_j)_s} \xi_s(Y).
\]
Thus, for each $j$,  the rational function 
$\xi=\sum_{s=1}^K \overline{(\alpha_j)_s}  \xi_s$ vanishes on
 $\cU.$ By hypothesis, $\cU_1\ne \emptyset$ and $\cU$ 
 is an open  free set. Hence, for each $n,$ the set $\cU_n$ is nonempty and open
 and $\xi$ vanishes identically on $\cU.$
 Hence $\xi$  is  identically zero since there are no rational identities \cite{Ber76}; 
 cf.~the definition of nc rational functions via matrix evaluations in \cite{HMV06}.
 The desired conclusion follows.  
\end{proof}

\begin{lemma}\label{l:beyondCheesy}
If the realization \eqref{eq:ratsSMRintro} is minimal and of size $N$ and $\cU$ 
 is a free open subset of $\dom r,$  then, for each $m\in\mathbb N,$
 there exists an $M$, $(A,X)\in \cU,$ a $w\in M_{m,M}(\C)$ and 
  an $H\in \bbS_M^\vg$   such that 
\begin{equation*}
V_{A,X,H,w}:=\{
	(I_N\otimes w) (\sum_i T_i \otimes H_i) R(A,X) (c\otimes I_M) v 
    \mid  v\in \C^M	\} = (\ran T)\otimes \C^m.
\end{equation*}
\end{lemma}

\begin{proof}
Let $K$ denote the dimension of $\ran T$ and $U$ a unitary matrix
 mapping $\ran T$ into the first $K$ coordinates of $\C^N.$
The entries $\eta_j$ of the $N\times 1$ matrix $R(a,x)c$ are linearly independent
nc rational functions by minimality of \eqref{eq:ratsSMRintro} and hence so are the entries of
 the $ \gv\, N\times1$ matrix
\[
Q(a,x,h):= \begin{pmatrix} h_1R(a,x)c \\ \vdots \\ h_\gv R(a,x)c\end{pmatrix}.
\]
Thus  there are  $\xi_j \in \ffthree$ such that 
\[
 \sum T_i h_i R(a,x)c = 
    \left [\begin{pmatrix} T_1 & \cdots & T_\gv  \end{pmatrix} \right ]  Q(a,x,h) 
 = U^*\, \col \left (
\xi_1, \cdots,  \xi_K, 0,  \cdots,  0 \right ).
\]
 Further, since the entries of $Q$ are linearly independent, the
 set  $\{\xi_1,\dots,\xi_K\}$ is linearly independent. By Lemma \ref{l:precheesy},
 for each positive integer $m,$ 
 there exists a positive integer $M$, a tuple 
 $(H,A,X)\in \bbS_M^\gv \times \cU_M$ and a matrix
 $w\in M_{M,m}(\C)$ such that the conclusion of Lemma \ref{l:precheesy} holds,
 completing the proof.
\end{proof}

\subsubsection{Proof of  Proposition \ref{l:character}}
\label{sssec:character}
Observe that, from equation \eqref{e:rxx} 
 it is evident that the inequality $R_T \succeq 0$ on $\cE$ 
 implies $r_{xx}$ is positive semidefinite on $\cE$, equivalently $r$ is
  convex in $x$ on $\cE$ by Proposition~\ref{p:HM98}.

 Now suppose $r_{xx}$ is positive semidefinite on $\cE.$ To prove that
 the inequality $R_T\succeq 0$ holds on $\cE$,  {\dis the variables},
 in the following way.  Let 
\[
x_i=\begin{pmatrix}
x_i^{1} & 0 \\ 0  & x^2_i
\end{pmatrix},
\quad
h_i=\begin{pmatrix}
 0 & k_i  \\ k_i^* &  0 
\end{pmatrix},
\quad
a_i=\begin{pmatrix}
a_i^{1} & 0 \\ 0 & a^2_i
\end{pmatrix},
\]
 where the  $x_i^j,$ $k_i$ and $a_i^j$ form a 
 $2(2\vg+\vh)$ collection of freely noncommuting variables.
In these coordinates the $(1,1)$ entry of $r_{xx}$ in
\eqref{e:rxx} equals
\begin{equation}
\label{e:endcharacter}
2 \left [c^* R(a^1,x^1)(\sum_i T_i k_i)\right ] \,  R(a^2,x^2) \,
 \left [(\sum_i T_i (k_i)^*) R(a^1,x^1) c\right ].
\end{equation}
We next apply Lemma \ref{l:beyondCheesy}.
 Given a positive integer $m$ and  $(A^2,X^2)\in \cE_m$, 
 choose $M$ and $(A^1,X^1)\in \cU_M,$ $w\in M_{m,M}(\C)$ and $H\in \bbS_M^\vg$
  satisfying the conclusion of Lemma \ref{l:beyondCheesy}. Thus
 $(A,X) = (A^1\oplus A^2, X^1\oplus X^2)\in \cE_{m+M}$ 
  and hence $r_{xx}(A,X)[H]\succeq 0.$ Choose $K=wH \in M_{m,M}(\C).$ 
   Substituting into  \eqref{e:endcharacter}
 and observing that $\{[\sum T_j \otimes K_j] R(A^1,X^1) (c\otimes I):v\in \C^n\}$
 spans $\ran T\oplus \C^m,$  it now follows that $R_T(A^2,X^2)\succeq 0.$  
\qed

\subsubsection{Proof of Theorem~\ref{t:character}}
 For item~\ref{i:open-and-convex}, Proposition~\ref{p:dompr-inx} says that
 $\domp r$ is open in $x$ and convex in $x.$  The forward
 direction of Proposition~\ref{l:character} says that
 $r$ is convex in $x$ on $\domp r.$

 The converse direction of Proposition~\ref{l:character}
 says, if $\cD$ is a \robust domain of convexity for $r,$
 then $R_T\succeq 0$ on $\cE.$ Thus $\cE\subseteq \domp r.$\hfill\qedsymbol

\subsection{Realizations for partial convexity}

\begin{proposition}\label{p:x-cvx}
The rational function   $r\in\fftwo$ of equation \eqref{eq:ratsSMRintro}
admits the realization
\begin{equation}\label{eq:undevelopedButterfly}
\begin{split}
r & = 
c^* (J-\sum S_ia_i)^{-1}c 
+  c^* (J-\sum S_ia_i)^{-1} \, \sum T_ix_i \, (J-\sum S_ia_i)^{-1} c \\
& \phantom{=\ } + c^* (J-\sum S_ja_j)^{-1} \sum T_ix_i\, 
\left (J-\sum T_j x_j -\sum S_k a_k\right)^{-1} 
\, \sum T_ix_i (J-\sum S_ia_i)^{-1}  c.
\end{split}
\end{equation}
\end{proposition}

We will refer to a realization of the form \eqref{eq:undevelopedButterfly}
 as a \df{caterpillar realization}.

\begin{proof}
Formula \eqref{eq:undevelopedButterfly} follows from a routine calculation.
\end{proof}

 Recall the definitions
 of $\VT$ and $\pi_a(\cD)$ from equations~\eqref{def:RsubT} and \eqref{e:pisuba}
 respectively.

\begin{theorem}[Wurzelschmetterlingrealisierung]\label{t:wurzelification}
Suppose $r\in\fftwo$ is symmetric 
 with \SMR %of size $N$ 
 as in equation \eqref{eq:ratsSMRintro}.

\begin{enumerate}[\rm (1)]
\item  \label{i:wurz-z}
   The set $\domkebp r$ is a \kebab domain of convexity for $r.$
\item 
  \label{i:wurz-d}
  If $\cD\subset\dom r$ is a \robust \kebab domain of convexity for $r,$
 then $\cD\subseteq \domkebp r$;  
\end{enumerate}

 Let  $\hatT_j = \VT^* T_j \VT$  and let $k$ denote the dimension of $\ran T.$
 There exists a  rational function $w(a)\in M_k(\ffa)$
 defined on $\pi_a(\domkeb r)$ and positive semidefinite on $\pi_a(\domkebp r);$
 rational functions $\ell_j(a)\in \ffa^k$ for $1\le j\le g,$ 
 that are %
 defined on $\domkeb r;$  %
 and a  rational function $f\!\!\!f (a,x)$  that is
    affine linear in $x$ and  defined on $\domkeb r$   
   such that, with
\begin{equation}
\label{e:ell}
 \ell(a,x) = \sum x_j \ell_j(a),
\end{equation}
\begin{enumerate}[resume*]%[\rm (1)]
 \item \label{i:wurz-a}
   if $(B,Y)\in\domkeb r;$ 
     then $I-(\sum T_j\otimes Y_j)w(B)$ is invertible and
\[
  r(B,Y)=\ell(B,Y)^*  w(B) \left (I-(\sum \hatT_i\otimes Y_i)w(B) \right)^{-1}  \ell(B,Y) + f\!\!\!f (B,Y);
\]
 \item \label{i:wurz-c} 
$
 \domkebp  r=\{(A,X) \in \domkeb r: w(A) \succeq 0 \text{ and }
      I - \sqrt{w(A)} \, [\sum \hatT_j\otimes X_j] \, \sqrt{w(A)} \succ 0\};
$
 and 
\begin{equation}\label{eq:wurscht}
r|_{\domkebp r}(a,x)=\ell(a,x)^*\sqrt {w(a)} \, \left (I-\sqrt {w(a)}\, \sum  \hatT_ix_i\, \sqrt {w(a)}\right)^{-1}
       \sqrt {w(a)}\,\ell(a,x) + f\!\!\!f (a,x);
\end{equation}
\item 
 \label{i:wurz-poly}
  If $r$ is a polynomial and $\cD$ is a \robust \kebab domain of convexity
  for $r,$ then 
 \begin{enumerate}[\rm (a)]
   \item $f\!\!\!f,$ $w,$ $\ell$ are also polynomials; 
   \item  $r$ has the representation, 
      \begin{equation}
       \label{e:wurzpoly}
            r(a,x) = \ell(a,x)^* w(a) \ell(a,x) +  f\!\!\!f (a,x),
       \end{equation}
    and hence $r$ is convex in $x$ on $\pi_a(\cD)\times \bbS^{\vg}$
 and  has degree at most two in $x.$
 \end{enumerate}
\end{enumerate}
 Conversely, any (rational) function of the form  \eqref{eq:wurscht}
 is convex in $x$ on the set $\domkebp r$  and   any polynomial of the form of equation
   \eqref{e:wurzpoly} is convex in $x$ on the free strip $\{A\in\bbS^\tg: w(A)\succeq 0\}\times \bbS^\vg.$
\end{theorem}

Given the symmetric realization \eqref{eq:ratsSMRintro},
express the matrices $T_j,S_j$ as block $2\times 2$ matrices with respect
 to the orthogonal decomposition $\ran T\oplus \ran T^\perp$ as
\begin{equation}
\label{def:That}
T_j=\begin{pmatrix} \hatT_j & 0 \\ 0 & 0\end{pmatrix},
\quad
S_j=\begin{pmatrix} S_{11}^j & S_{12}^j \\ S_{12}^{i*} & S_{22}^j\end{pmatrix}, \quad
J=\begin{pmatrix} J_{11} & J_{12} \\ J_{12}^* & J_{22}\end{pmatrix}.
\end{equation}

\begin{proof}[Proof of Theorem~\ref{t:wurzelification}]
  By definition, $\domkebp r$ is convex in $x$
  and a subset of $\domp r.$ Thus, since $r$ is convex
 in $x$ on $\domp r,$ it is also convex in $x$ on 
 $\domkebp r.$ Thus item~\ref{i:wurz-z} holds.

 If $\cD\subseteq \dom r$ is full \kebab domain of
 convexity for $r,$ then $\cD$ is a \robust  domain 
 of partial convexity for $r.$ Hence, by 
 Theorem~\ref{t:characterINTRO}, $\cD\subseteq \domp r.$
 If $(A,X)\in \cD,$ then $(A,0)\in \cD,$ since
 $\cD$ is a \kebab set. Thus both $(A,X)$ and $(A,0)\in \domp r$
 and hence $(A,X)\in \domkebp r,$ proving item~\ref{i:wurz-d}.

By Proposition \ref{p:x-cvx},
$r$ admits the caterpillar realization
\eqref{eq:undevelopedButterfly} whose resolvent,
\begin{equation*}
R(a,x)=\begin{pmatrix}
J_{11}-\sum \hatT_j x_j-\sum S_{11}^ja_j & J_{12}-\sum S_{12}^ja_j \\
J_{12}^*-\sum S_{12}^{j*} a_j & J_{22}-\sum S_{22}^ja_j,
\end{pmatrix}^{-1}
\end{equation*}
 is defined on the domain of $r.$ 
 We obtain a free rational function  $W(a)=R(a,0)\in \ffa.$ Let 
 $w(a)=\VT^* R(a,0)\VT$ denote the (block) $(1,1)$-entry of $W(a).$ 
  Likewise the domain of the  rational function
\[
\ell(a,x)= \VT^* \sum T_ix_i W(a) c
\]
contains $\dom W.$ %

 Suppose $(A,X)\in \domkeb r.$ 
 Thus $(A,0),(A,X)\in \dom r,$  and hence
\begin{equation}
 \label{e:ref1-1}
\begin{split}
 R^{-1}(A,X)W(A)& =  \left (J-\sum T_j\otimes X_j - \sum S_k\otimes A_k\right) W(A)\\
  &  = I - \left(\sum T_j \otimes X_j\right)W(A) \\
  & = \begin{pmatrix} \left (I-\sum \hatT_j\otimes X_j\right )  w(A) & *\\0 & I \end{pmatrix}.
\end{split}
\end{equation}
It follows that $I-(\sum \hatT_j\otimes X_j) w(A)$ is invertible whenever 
$(A,0), (A,X)\in \dom r$, establishing the first half of item~\eqref{i:wurz-a}.
 Moreover, in that case, from equation~\eqref{e:ref1-1},
\[
 R(A,X)  = W(A) \, \begin{pmatrix} \left (I-\sum \hatT_j\otimes X_j\right )  w(A) & *\\0 & I \end{pmatrix}^{-1}
 = \begin{pmatrix} w(A) \left (I-\sum \hatT_j\otimes X_j\right )  w(A) & *\\0 & I \end{pmatrix}^{-1}
\] 
 and thus
\begin{equation*}
%\label{e:RTvwA}
R_T(a,x) = \VT^* R(a,x)\VT =  w(a) \left (I-(\sum \hatT_ix_i)w(a) \right)^{-1}.
\end{equation*}
 Letting $f\!\!\!f$ denote the affine linear in $x$ term
 from the caterpillar realization of equation~\eqref{eq:undevelopedButterfly},
\[
 r(A,X)=\ell(A,X)^*  w(A) \left (I-(\sum \hatT_i\otimes X_i)w(A) \right)^{-1}  \ell(A,X) + f\!\!\!f (A,X),
\]
when $(A,X)\in \domkeb r,$ proving item \eqref{i:wurz-a}.

 Given square matrices $P$ and $Q$ of the same size, 
 the eigenvalues of $PQ$ and $QP$ are the same.
 Now suppose $(X,A)\in \domkeb r$ and $w(A)\succeq 0$ and let 
  ${\tt{T}}=\sum \hatT_i\otimes X_i.$ Choosing
 $P={\tt{T}}\sqrt{w(A)}$ and $Q=\sqrt{w(A)},$ it follows
 that ${\tt{T}}w(A)$ and $\sqrt{w(A)}{\tt{T}}\sqrt{w(A)}$
 have the same eigenvalues. Thus, in view of item~\eqref{i:wurz-a}, if
 $I-\sqrt{w(A)}{\tt{T}}\sqrt{w(A)}\succeq 0,$ 
  then $I-\sqrt{w(A)}{\tt{T}}\sqrt{w(A)}\succ0.$ Hence
\[
 R_T(A,X)= w(A)(I-{\tt{T}}w(A))^{-1} = 
  \sqrt{w(A)} \left(I-\sqrt{w(A)}{{\tt{T}}}\sqrt{w(A)} \right)^{-1} \sqrt{w(A)}\succeq 0
\]
 and therefore $(A,X)\in  \domp r.$ The assumption
 $R_T(A,0) = w(A)\succeq 0$ is equivalent to $(A,0)\in \domp r.$
 Hence $(A,X)\in \domkebp r.$ 

Conversely, if $(A,X)\in\domkebp r,$
 then $w(A)\succeq 0$ and, since $\domkebp r$
 is convex in $x$ and $(A,0)\in\domkebp r,$
 for each $0\le t\le 1,$ the matrix
  $I-t {\tt{T}}w(A)$ is invertible
 and hence so is $M(t)=I-\sqrt{w(A)}{\tt{T}}\sqrt{w(A)}.$
 Since $M(0)$ is positive and $M(t)$ is invertible 
 and self-adjoint for $0\le t\le 1,$ it follows
 that $M(1)\succ 0$ and the proof of item~\eqref{i:wurz-c}
 is complete.

 In the case $r$ is a polynomial,  
 $R(a,x)$ is globally defined (has no singularities) and is therefore a (matrix-valued) polynomial
 by \cite[Corollary 3.4]{klep-volcic}. Hence both $w(a)$ and $\ell(a,x)$
 are polynomials.  By hypothesis, there is a free open set $\cU\subset \cD$
 with $\cU_1\ne \emptyset.$ Choose a point $(\tta,\ttx)\in \cU_1 
  \subset \mathbb R^{\tg}\times\mathbb R^{\vg}$ and consider the polynomial
 $q(a,x)=r(a-\tta,x).$  Let $\cD^\prime =\{(A-\tta I,X):(A,X)\in \cD\}.$
 If $(A,X)\in \cD^\prime,$ then $(A-\tta I,X)\in \cD$ and hence
 $(A-\tta I,0)\in \cD$ and finally $(A,0)\in \cD^\prime.$ Thus
 $\cD^\prime$ is a \kebab domain of partial convexity for $q.$ 
 Hence, without loss of generality, we assume from the outset
 that $(0,0)\in \cD.$ Then  $w(0)=\VT^* R(0,0)\VT$ is positive semidefinite
 by Theorem \ref{t:character}   since 
 we have now convexity in $x$ in a neighborhood of $0.$ Next
 $R(0,0)=J^{-1}=J$ and so $w(0)=J_{1,1}\succeq 0.$ Since $r$ is a polynomial
 (and the realization is minimal), $TJ$ is (jointly) nilpotent
 by \cite[Corollary 3.4]{klep-volcic}. But
\[
 TJ =\begin{pmatrix} \hat T &0\\0&0 \end{pmatrix} \,
 \begin{pmatrix} J_{11} & J_{12}\\ J_{12}^* & J_{22}\end{pmatrix}
 =\begin{pmatrix} \hat T J_{11}& \hat T J_{12}\\0&0\end{pmatrix},
\]
whence $\hat T J_{12}$ is (jointly) nilpotent. Thus
 $Y=\sqrt{J_{11}}T_j \sqrt{J_{11}}$ is self-adjoint and nilpotent
 and hence $0.$  Thus, from equation \eqref{eq:wurscht}, $r$
 has the representation of equation \eqref{e:wurzpoly}. From this
 representation it is immediate that $r$ has degree (at most) two in $x$
 and is convex in $x$ on the set $\{(A,X):w(A)\succeq 0\},$ which includes
 $\pi_a(\cD)\times \bbS^{\vg}.$
\end{proof}

\begin{corollary}%[Entwurzelteschmetterlingrealisierung]
\label{c:notwurzelification}
Let $\cD$ be a  \kebab set.
Let $r\in\fftwo$ be a nc rational function in two
classes of variables $x=(x_1,\ldots,x_\gv)$ and
$a=(a_1,\ldots,a_{\tg})$.
Let $r$ have a \SMR \eqref{eq:ratsSMRintro}. 
Consider the matrices in block form based on $\ran T$ in
equation \eqref{def:That} and let $k$ denote the 
 dimension of $\ran T.$

If  $J_{22}$ is invertible, then the function $r$ is convex in $x$ on $\cD$ if 
and only if 
there exists a rational function $\ell(a,x)\in \fftwo^{k\times 1}$ that
is linear in $x$, and a rational function $m(a)\in \fftwo^{k\times k}$ such
that
\[
r=\ell(a,x)^* \, \left (m(a)-\sum \hatT_ix_i\right)^{-1}\, \ell(a,x) + f\!\!\!f (a,x),
\]
where $f\!\!\!f (a,x)\in \fftwo$ is affine linear in $x$,
and the  resolvent $(m(a)-\sum \hatT_ix_i)^{-1}$
is positive on a dense subset of $\cD_n$ for large $n$.
\end{corollary}

\begin{proof}
This result follows by using the Schur complement form for the inverse of a block matrix in  
Proposition \ref{p:x-cvx},   the positivity
condition follows from 
Proposition \ref{l:character}. 
\end{proof}

\section{A polynomial factorization}
\label{sec:factor}
In this section we introduce an auxiliary operation $\sE$
on both matrices and polynomials and in Theorem 
\ref{l:FstarFisM} provide a decomposition of symmetric
polynomials $\pp\in M_2(\C\axy)$ for which $\sE{\pp}$
is (matrix) positive. This result is a key ingredient
 in the proof of Theorem \ref{t:introxyconvexp}, which
 appears in Section~\ref{sec:xypolys}, characterizing
$xy$-convex polynomials.

Given a pair of block $2\times 2$ matrices
$A=(A_{i,j})$ and $B=(B_{i,j})$ define
\[
 A\circledast B = \begin{pmatrix} A_{i,j}\otimes B_{i,j} \end{pmatrix}.
\]
Thus \df{$A\circledast B$} is a mix of Schur product ($\ast$) and tensor product ($\otimes$).
 It is known as the Khatri-Rao product. 
Let $V_1=\begin{pmatrix} I \\ 0 \end{pmatrix}$
  and $V_2 =\begin{pmatrix} 0 \\ I \end{pmatrix}$ with respect to the
 block decomposition of $A$ and define $W_1,W_2$ similarly with respect to the
 block decomposition of $B$. Let
\[
 E =\begin{pmatrix} V_1\otimes W_1 & V_2\otimes W_2 \end{pmatrix}.
\]

\begin{lemma}
\label{l:astvtensor}
 With notation as above,  $A\circledast B = E^* [A\otimes B] E$.
\end{lemma}

\begin{proof}
 Note that 
\[
  E^* [A\otimes B] E = \left(  (V_j^*\otimes W_j^*) [A\otimes B] (V_k\otimes W_k)
 \right)_{j,k=1}^2
\]
 and $(V_j^*\otimes W_j^*) [A\otimes B] (V_k\otimes W_k) = A_{jk}\otimes B_{jk}$.
\end{proof}

Let, for $j=1,2$, 
\[
 s_j=\begin{pmatrix} s_{j,0} & s_{j,1} \\ s_{j,1}^* & s_{j,2}\end{pmatrix},
\]
where $\{s_{j,k}: 1\le j\le 2, \, 0\le k\le 2\}$ are freely noncommuting
variables with  $s_{j,0}$ and $s_{j,2}$ symmetric; that is $s_{j,k}^* = s_{j,k}$
 for $k=0,2$.    For notational purposes, let
\[
 s_0 = I_2 = \begin{pmatrix}  1 & 0\\0& 1 \end{pmatrix}.
\]
Suppose $p= \sum_{j,k=0} p_{j,k} \, x_j x_k,$ is a $2\times 2$ symmetric matrix 
  polynomial of degree (at most) two in
 two symmetric variables $x=(x_1,x_2),$ 
where, for notation purposes,  $\sx_0=1$ (the unit in $\C\ax$),
  each $p_{j,k}\in M_2(\C)$  and $p_{j,k}^*=p_{k,j}.$
 Let \df{$\sE{p}$} denote the matrix polynomial in the six
variables $\{s_{j,0},s_{j,1},s_{j,2}: 1\le j\le 2\}$ defined by
\[
 \sE{p}(s) = \sum_{j,k=0}^2 p_{j,k} \circledast s_js_k.
\]
 Such a polynomial is naturally evaluated 
at a pair of block $2\times 2$ symmetric matrices,
\begin{equation}
\label{eq:S2x2}
 S_j= \begin{pmatrix} S_{j,0}& S_{j,1}\\S_{j,1}^* & S_{j,2}\end{pmatrix} \in 
  M_{\mu}(\C)\otimes M_2(\C)
\end{equation}
using $\circledast$ via
\[
 \sE{p}(S) = \sum_{j,k=0}^2  p_{j,k}\circledast S_j S_k \in  M_{\mu}(\C) \otimes M_2(\C).
\]
By contrast,
\[
 p(S) =  \sum_{j,k=0}^2 p_{j,k}\otimes S_jS_k \in M_2(\C)\otimes M_{\mu}(\C)\otimes M_2(\C).
\]
However, $p$ and $\sE{p}$ are closely related, as the following lemma
describes. Its proof is similar to that of Lemma \ref{l:astvtensor}.

\begin{lemma}
 \label{l:james-ast}
 With notations as above,
  \[
   \sE{p}(S) = E^*  \left (\sum_{j,k=0}^2 p_{j,k}\otimes S_jS_k\right)\ E 
     = E^* p(S) E.
\]
In particular, if $p(S)\succeq 0$, then $\sE{p}(S)\succeq 0$ too.
 \end{lemma}

Theorem \ref{l:FstarFisM} is the main result of this section. 

\begin{theorem}
\label{l:FstarFisM}
 Suppose $\pp(x)$ is a symmetric $2\times 2$ polynomial of degree at most two in the 
symmetric variables $x=(x_1,x_2).$ If $\sE{\pp}(S)\succeq 0$ for all positive
integers $m,n$ and 
pairs 
$S=(S_1,S_2)\in \bbS_{n+m}^2$ of $2\times 2$ block symmetric matrices,
then there exists an $N\le 12$ and $q_0,q_1,q_2\in M_{N, 2}(\C)$ such that 
\begin{equation*}
%\label{e:FFM1}
\begin{split}
 q_j^* q_k  &= \pp_{j,k}, \ \  1\le j, k\le 2, \\
 q_0^*q_k + q_k^* q_0 &= \pp_{k,0}+\pp_{0,k}, \ \ k=1,2, 
\end{split}
\end{equation*}
\begin{equation}
\label{e:FFM2}
 (q_0^* q_0)_{1,1} = (\pp_{0,0})_{1,1}, \ \ (q_0^*q_0)_{2,2} = (\pp_{0,0})_{2,2}.
\end{equation}
In particular, letting $q$ denote the 
affine linear  polynomial 
  $q=\sum_{j=0}^2  q_jx_j \in \C\langle x\rangle^{N\times 2}$,
 there is an $r_1\in \C$ such that 
\[
 \pp= q^*q  +  r, \quad\text{where }r= \begin{pmatrix} 0 & r_1 \\ r_1^* & 0 \end{pmatrix}.
\]
\end{theorem}

 The remainder of this section is devoted to the proof of
 Theorem \ref{l:FstarFisM}.
  Let $\{e_1,e_2\}$ denote the standard orthonormal basis for $\C^2$
 with resulting matrix units  $e_ae_b^*$ for $1\le a,b\le 2.$
 Let $\as_k$ denote the words in $\sx_1,\sx_2$ of length at most $k$. 
 Thus $\as_1 =\{\sx_0,\sx_1,\sx_2\},$ 
 where, as above,  $\sx_0=1.$
 We will view $\C^3$ as the span of $\as_1$ with $\as_1$ as an orthonormal
 basis and $M_3(\C)$ as matrices indexed by $\as_1 \times \as_1.$
 In this case $x_jx_k^*$ are the matrix units.\looseness=-1

Let \df{$\msS$}  denote the subspace of $M_2(\C)\otimes M_3(\C)$ consisting of 
matrices 
\[
 T = \begin{pmatrix} T_{\alpha,\beta}\end{pmatrix}_{\alpha,\beta\in \as_1},
\]
 where $T_{\alpha,\beta}\in M_2(\C)$ satisfy, for $\beta\in \as_1,$
\[
 T_{\beta,x_0}=T_{x_0,\beta}, \ \ \ T_{x_0,x_0}\in \Span\{e_1e_1^*,e_2e_2^*\}.
\]
 Thus $T_{x_0,x_0}$ is diagonal and 
 $\msS$ is an \df{operator system}; that is, a 
  self-adjoint subspace of $M_2(\C)\otimes M_3(\C)$ that contains the identity. 

 Define $\psi:\msS\to M_2(\C)$ by \index{$\psi$}
\begin{equation}
 \label{def:psi}
 \psi \begin{pmatrix} T_{\alpha,\beta}\end{pmatrix}
  =\sum_{\alpha,\beta\in\as_1} \rho_{\alpha,\beta}\ast T_{\alpha,\beta}
  = \sum_{\alpha,\beta\in\as_1} \rho_{\alpha,\beta}\circledast T_{\alpha,\beta}.
\end{equation}

\begin{proposition}
 \label{p:psi-is-cp}
   The mapping $\psi$ of equation \eqref{def:psi}  is completely positive (cp).
\end{proposition}

\begin{proof}
To prove that $\psi$ is cp, 
let a positive integer $n$ and positive definite
 $Z\in M_n(\C)\otimes \msS$ be given. In particular,
\[
 Z=\begin{pmatrix} Z_{\alpha,\beta} \end{pmatrix}_{\alpha,\beta\in \langle \sx_1,\sx_2\rangle_1},
\]
 where $Z_{\alpha,\beta}  
 =\begin{pmatrix} (Z_{\alpha,\beta})_{a,b}\end{pmatrix}_{a,b=1}^2 \in M_n(\C)\otimes M_2(\C),$ 
 $(Z_{\alpha,\beta})_{a,b}\in M_n(\C)$ and 
\[
 Z_{x_0,\beta} = Z_{\beta,x_0}, \ \ \ 
   Z_{x_0,x_0}  = \sum_{a=1}^2 (Z_{x_0,x_0})_{a,a} \otimes e_a e_a^*.
\] 
  Since $Z$ is positive definite, $Z_{x_0,\alpha}^*=Z_{x_0,\alpha}$ and
  letting $\Theta= Z_{x_0,x_0}^{-1},$
\[
0\preceq  \begin{pmatrix} Z_{\alpha,\beta} - Z_{{\alpha},x_0} \Theta Z_{x_0,\beta} 
    \end{pmatrix}_{|\alpha|=|\beta|=1}
  = GG^* = \begin{pmatrix} G_{\alpha} G_{\beta}^* \end{pmatrix}_{|\alpha|=|\beta|=1},
\]
 for some $m$ and matrices 
\[
 G_\alpha = \begin{pmatrix} (G_\alpha)_{a,j} \end{pmatrix}_{a,j=1}^2 
   \in M_{n,m}(\C)\otimes M_2(\C).
\]
In particular,  for $1\le a,b\le 2,$
\[
(Z_{\alpha,\beta})_{a,b} 
  - \left [  Z_{\alpha,x_0} \, 
 \begin{pmatrix} \Theta_{1,1} & 0  \\ 0 & \Theta_{2,2}   \end{pmatrix} \, 
     Z_{x_0,\beta}\right ]_{a,b} 
  =    \sum_{j=1}^2 (G_\alpha)_{a,j}\, (G_\beta)_{b,j}^*,
\]
 where $\Theta_{j,j} = (Z_{x_0,x_0})_{j,j}^{-1}.$ Thus, for $|\alpha|=1=|\beta|,$
\[
  \sum_{j=1}^2 (Z_{\alpha,x_0})_{a,j} \Theta_{j,j} (Z_{x_0,\beta})_{j,b} 
   + \sum_{j=1}^2 (G_\alpha)_{a,j}\, (G_\beta)_{b,j}^*
   = (Z_{\alpha,\beta})_{a,b}.
\]

Let 
\[
 \Psi = \begin{pmatrix} \Psi_{1,1} & 0 \\ 0 & \Psi_{2,2} \end{pmatrix}
  \in M_{n+m}(\C)\otimes M_2(\C),
  \quad\text{where}\quad
\Psi_{a,a} = \begin{pmatrix} (Z_{x_0,x_0})_{a,a} & 0 \\ 0 & I_m \end{pmatrix}
  \in M_{n+m}(\C).
\]
 Let, for $j=1,2,$
\begin{equation}
\label{eq:Wj}
 W_j = \begin{pmatrix} (W_j)_{a,b} \end{pmatrix} \in M_{n+m}(\C)\otimes M_2(\C),
 \quad \text{where}\quad
 (W_j)_{a,b} =\begin{pmatrix} (Z_{x_0,x_j})_{a,b} & (G_{x_j})_{a,b} \\ 
   (G_{x_j})_{b,a}^* & 0 \end{pmatrix} \in M_{n+m}(\C).
\end{equation}
 Since $Z_{\alpha,x_0}=Z_{x_0,\alpha}$ is self-adjoint, so is $W_j.$  By construction, 
\[
 (W_j \Psi^{-1} W_k)_{a,b}  = 
  \begin{pmatrix} (Z_{x_j,x_k})_{a,b} & * \\ * & * \end{pmatrix} \in M_{n+m}(\C).
\]

Let
\[
  W =\begin{pmatrix} \Psi  & W_1 & W_2 \\ W_1 & W_1\Psi^{-1} W_1 & W_1\Psi^{-1} W_2\\
   W_2 & W_2\Psi^{-1} W_1 & W_2\Psi^{-1} W_2 \end{pmatrix} 
   \in M_{n+m}(\C)\otimes \msS
\]
and let   $V\in M_{2(n+m),2n}(\C)$ denote the isometry  whose adjoint is
 \[
V^* = \begin{pmatrix} I_n & 0 & 0 & 0\\ 0&0&I_n&0 \end{pmatrix} 
  \in M_{2n,2(n+m)}(\C),
\]

 From the definition~\eqref{def:psi} of $\psi$ (and letting $\psi$ also denote
 its ampliations $\psi\otimes I_\ell,$ where $I_\ell$ is the identity on $M_\ell(\C)$),
\begin{equation}
\label{e:psiW}
 \psi(W) = \rho_{x_0,x_0}\circledast \Psi + \rho_{x_0,x_1}\circledast W_1 + 
  \rho_{x_0,x_2}\circledast W_2 + \sum_{j,k=1}^2 \rho_{x_j,x_k}\circledast W_j \Psi^{-1} W_k.
\end{equation}
 By definition of the $\circledast$ operation, given
% given a positive integer $\nu,$ and
\begin{equation} \label{e:more}
\begin{split}
 R& =\begin{pmatrix} R_{1,1}&R_{1,2}\\R_{2,1}&R_{2,2}\end{pmatrix} \in M_{n+m}(\C)\otimes M_2(\C)\\
 R_{i,j} & = \begin{pmatrix} R_{i,j}^{1,1} & R_{i,j}^{1,2} \\R_{i,j}^{2,1} & R_{i,j}^{2,2} \end{pmatrix} 
   \in M_n(\C)\oplus M_m(\C)\\
 \tau & = \begin{pmatrix} \tau_{1,1} & \tau_{1,2}\\ \tau_{2,1}&t_{2,2}\end{pmatrix}\in M_2(\C)
\end{split}
\end{equation}
 we have  $\tau\circledast R =\begin{pmatrix} \tau_{i,j}  R_{i,j} \end{pmatrix}$
and hence
\[
  V^* \, [ \tau \circledast R] \, V
 =  \begin{pmatrix} \tau_{i,j}  R_{i,j}^{1,1} \end{pmatrix} = \tau\circledast \widetilde{R},
\]
 where $\widetilde{R} = \begin{pmatrix} R_{i,j}^{1,1} \end{pmatrix}_{i,j=1}^2.$
Hence,
\[
\begin{split}
 V^*  \, \left [\rho_{x_0,x_0}\circledast \Psi\right ] \, V
  & = \rho_{x_0,x_0}\circledast Z_{x_0,x_0}\\
V^*  \, \left [ \rho_{x_0,x_j} \circledast W_j  \right ] \, V
  & = \rho_{x_j,x_k} \circledast Z_{x_j,x_k} \\
V^*\,  \left [\rho_{x_j,x_k}\circledast W_j\Psi^{-1}W_k \right]
  \, V & = \rho_{x_j,x_k} \circledast Z_{x_j,x_k}.
\end{split}
\]
Thus, from equation~\eqref{e:psiW} 
\[
 V^* \psi(W) V = \psi(Z).
\]
Hence, to prove $\psi(Z)\succeq 0$ it suffices to show $\psi(W)\succeq0.$

 With $R$ and $\tau$ as in equation~\eqref{e:more}, given a  block diagonal matrix
\[
 D=\begin{pmatrix} D_1 & 0 \\0 & D_2 \end{pmatrix} \in M_{n+m}(\C)\otimes M_2(\C),
\]
 we have
\[
\begin{split}
D\,  \left [ \tau\circledast R \right ] \, D 
& = \begin{pmatrix} D_1 & 0 \\0 & D_2 \end{pmatrix} \, 
   \begin{pmatrix} \tau_{i,j} R_{i,j} \end{pmatrix}  \, 
 \begin{pmatrix} D_1 & 0 \\0 & D_2 \end{pmatrix}\\
& = \begin{pmatrix} \tau_{i,j}  D_i R_{i,j} D_j \end{pmatrix}
 = \tau\circledast (DRD).
\end{split}
\]
Hence, 
 $S_j =\Psi^{-\frac12} W_j \Psi^{-\frac12} \in M_{n+m}(\C) \otimes M_2(\C)$ are self-adjoint and 
\[
\begin{split}
\Psi^{-\frac12} \,  \psi(W)\, \Psi^{-\frac12}
 & = \sum_{j,k}  \Psi^{-\frac12}  \, \left [
 \rho_{j,k}\circledast W_{j,k}\right ]\,  \Psi^{-\frac12}\\
 & =   \sum_{j,k} \rho_{j,k}\circledast S_jS_k 
   =  \mathcal{E}\rho(S).
\end{split}
\]
By hypothesis $\mathcal{E}\rho(S)\succeq 0$
and hence  $\psi(W)\succeq 0.$ Thus
 $\psi(Z)\succeq0$ under the extra
 assumption that $Z\succ0.$

Now suppose  $Z\in M_n(\C)\otimes \msS$ is positive semidefinite.
Since the identity is contained in $M_n(\C)\otimes \msS,$
for each $\epsilon >0,$ the matrix $Z+\epsilon I$
 is positive definite and in $M_n(\C)\otimes \msS.$
 Thus, by what has already been proved, $\psi(Z+\epsilon I)\succeq 0$
 and hence, by letting $\epsilon$ tend to $0,$ it follows
 that $\psi(Z)\succeq0$ and the proof is complete.
\end{proof}

\begin{proof}[Proof of Theorem~\ref{l:FstarFisM}]
 Since, by Proposition \ref{p:psi-is-cp},  $\psi$ is cp it extends, 
 by the Arveson Extension Theorem \cite[Theorem 7.5]{paulsen}, 
 to a  cp map $\varphi:M_2(\C)\otimes M_3(\C)\to M_2(\C)$. By a well-known result
of Choi \cite[Theorem 3.14]{paulsen}, its Choi matrix 
\[
 C_{\varphi} = \sum_{j,k=0}^2 \sum_{a,b=1}^2 [e_a e_b^*\otimes \sx_j \sx_k^*]\, \otimes \, 
   [\varphi(e_a e_b^*\otimes \sx_j \sx_k^*)] \in M_2(\C)\otimes M_3(\C)\otimes M_2(\C)
\]
is positive semidefinite.  In particular, $C_{\varphi}$ factors as
$F^*F$ where, 
\[
 F= \sum_{a=1}^2 \sum_{j=1}^3    \etwo_a^*\otimes \ethree_j^*\otimes F_{j,a}
\]
for some $N$ ($\le 12$) 
 and $N\times 2$ matrices $F_{j,a}$ and, in particular,
\begin{equation}
\label{e:ifcp1}
 F_{j,a}^*F_{k,b} = \varphi(e_a e_b^*\otimes \sx_j \sx_k^*).
\end{equation}

For
$
 q_j =\begin{pmatrix} F_{j,1} \etwo_1 & F_{j,2}\etwo_2 \end{pmatrix} 
  \in M_{N,2}(\C),
  $ we have
$
 q_j^* q_k = \begin{pmatrix}  e_a^* F_{j,a}^* F_{k,b}e_b \end{pmatrix}_{a,b=1}^2 
  \in M_2(\C).$
 So, using \eqref{e:ifcp1}, for $a=1,2$,
\[
(\pp_{0,0})_{a,a} =  \left (\pp_{0,0} \circledast e_a e_a^*\right )_{a,a}
 =  \psi(e_a e_a^*\otimes \sx_0 \sx_0^*)_{a,a} 
 =   \varphi(e_a e_a^*\otimes \sx_0 \sx_0^*)_{a,a} 
  = e_a^* F_{0,a}^* F_{0,a} e_a = (q_0^*q_0)_{a,a}.
\]
Hence equation \eqref{e:FFM2} holds.  Next, for $\ell =1,2$ and $1\le a,b\le 2$,
\[
\begin{split}
 (\pp_{0,\ell} + \pp_{\ell,0})_{a,b}  
  &= e_a^* \left [(\pp_{0,\ell} + \pp_{\ell,0})\circledast e_a e_b^*\right ] e_b
  = e_a^* \psi\left (e_a e_b^*\otimes (\sx_0 \sx_\ell^*+\sx_\ell\sx_0^* )\right )e_b \\
  &=  e_a^* \varphi\left(e_a e_b^*\otimes (\sx_0 \sx_\ell^*+ \sx_\ell \sx_0^*)\right )e_b 
  = e_a^* [F_{0,a}^* F_{\ell,b}+F_{\ell,a}^* F_{0,b}]e_b \\
 &= (q_0^* q_\ell + q_\ell^* q_0)_{a,b}.
\end{split}
\]
Thus $q_0^*q_\ell + q_\ell^*q_0  = \pp_{0,\ell}+\pp_{\ell,0}.$

Finally, we see that $q_j^*q_k  = \pp_{j,k}$ (for $1\le j,k\le 2$) by computing,
 for $1\le a,b\le 2,$ 
\[
\begin{split}
(\pp_{j,k})_{a,b} &= e_a^* [\pp_{j,k}\circledast e_a e_b^*] e_b 
= e_a^*\psi(e_a e_b^*\otimes \sx_j \sx_k^*)e_b \\
&= e_a^*\varphi(e_a e_b^*\otimes \sx_j \sx_k^*)e_b 
=  e_a^*F^*_{j,a}F_{k,b}e_b 
=  (q_j^*q_k)_{a,b}. \qedhere
\end{split} 
\]
\end{proof}

\section{The characterization of $xy$-convex polynomials}
\label{sec:xypolys}
 In this section we prove Theorem \ref{t:introxyconvexp}. In 
 Subsection \ref{sec:xyimpliesxandy} it is established that
 $xy$-convex polynomials are biconvex (convex in $x$ and $y$
 separately). Two applications of equation \eqref{e:wurzpoly} of 
 Theorem \ref{t:wurzelification} then significantly reduce the complexity of the
 problem of characterizing $xy$-convex polynomials. The notion of the $xy$-Hessian
 of a polynomial  is introduced in Subsection
 \ref{sec:xyHess},  where a \df{border vector-middle matrix}
 (see for instance \cite{HKMfrg})   representation for this Hessian is 
 established. Further, it is shown 
 that this middle matrix is positive for $xy$-convex polynomials.
 The proof of Theorem \ref{t:introxyconvexp} concludes in 
  Subsection \ref{sec:proofxyconvexp} by combining positivity 
 of the middle matrix and  Theorem \ref{l:FstarFisM}.

\subsection{$xy$-convexity implies biconvexity}
\label{sec:xyimpliesxandy}
The notion of $xy$-convexity for polynomials has a convenient concrete reformulation.

\begin{proposition}
 \label{p:xy-convexp-alt}
  A triple $((X,Y),V)$ is an $xy$-pair if and only if, up to unitary equivalence,
  it has the block form
\begin{equation}
\label{e:xypairform}
 X=\begin{pmatrix} X_0 & A&0\\ A^* & * & * \\ 0 & * & * \end{pmatrix}, \ \
 Y= \begin{pmatrix} Y_0 & 0&C\\ 0& *&*\\ C^* & * & * \end{pmatrix}, \ \ 
 V= \begin{pmatrix} I& 0&0\end{pmatrix}^*.
\end{equation}
 Thus, a polynomial $p(x,y)\in M_\mu(\C\axy)$ is $xy$-convex 
  if and only if for each $xy$-pair $((X,Y),V)$ of the form of 
 equation \eqref{e:xypairform}, we have
\[
 \IV^*p(X,Y)\IV - p(X_0,Y_0) \succeq 0.
\]
\end{proposition}

\begin{proof}
 Observe that $(X_0,Y_0)=V^*(X,Y)V$ and  $((X,Y),V)$
  is an $xy$-pair; that is $V^* YXV = V^* Y VV^*  X V$. 
 Thus, if $p$ is $xy$-convex, then
\[
 0\preceq  \IV^* p(X,Y)\IV - p(V^*(X,Y)V)  = \IV^* p(X,Y) \IV -p(X_0,Y_0).
\]

To establish the reverse implication, given an  $xy$-pair $((X,Y),V),$
 decompose the space $(X,Y)$ act upon as $\range V \oplus (\range V)^\perp$
 and note that, with respect to this orthogonal decomposition,
  $X$ and $Y$ have the block form 
\[
X=\begin{pmatrix} X_0 & \alpha \\ \alpha^* & \beta \end{pmatrix}, \ \
Y = \begin{pmatrix} Y_0 & \gamma \\ \gamma^* & \delta \end{pmatrix},
\]  where $X_0,Y_0,\beta,\delta$ are hermitian. 
 The relation $V^* YXV = V^* Y VV^*  X V$ implies  
 $\alpha \gamma^*=0.$
  But then, $\alpha$ and $\gamma$ are, up to unitary equivalence,
 of the form $\begin{pmatrix}A & 0\end{pmatrix}$
 and $\begin{pmatrix}0 & C\end{pmatrix}$, respectively. 
\end{proof}

Consider the following list of monomials: \index{$\listL$}
\begin{equation}
 \label{eq:list}
  \listL= \{1,x,y,x^2,y^2,xy,yx,xy^2,y^2x,x^2y,yx^2,xyx,yxy, xyxy,yxyx,xy^2x,yx^2y\}.
\end{equation}

\begin{proposition}
 \label{p:HHLM}
   If $p\in \C\langle x,y\rangle$ is convex in both $x$ and $y$ (separately), then 
 $p$ has degree at most two in both $x$ and $y$ (separately) and
   $p$ contains no monomials of the form $x^2y^2$ or $y^2x^2$, only the monomials in 
 the set $\listL.$
\end{proposition}

\begin{proof}
 The degree bounds follow from Theorem \ref{t:wurzelification}. The representation
 of $p$ in \eqref{e:wurzpoly} and that of $\ell$ in \eqref{e:ell}
  imply $p$ does not contain the monomials $x^2y^2$ and $y^2x^2.$
\end{proof}

 Let \df{$\PL$} denote the $\C$-vector space with basis $\listL$
 of equation \eqref{eq:list}.

\begin{lemma}
\label{l:xyconveximpliesbiconvex}
 If $p\in \C\langle x,y\rangle$ is $xy$-convex, 
 then $p$ is convex in  both $x$ and $y$. Hence $p\in \PL.$
\end{lemma}

\begin{proof}
 Given $(X_1,Y)$ and $(X_2,Y)$, let 
$V=\frac{1}{\sqrt{2}}\begin{pmatrix} I & I\end{pmatrix}^T$
 and note $((X_1\oplus X_2, Y\oplus Y),V)$ is an $xy$-pair. Since $p$ is $xy$-convex,
\[
 p\Big (\frac{X_1+X_2}{2},Y\Big ) = p(V^*(X,Y)V) \preceq V^* p(X,Y)V 
   = \frac{1}{2} \big(p(X_1,Y)+p(X_2,Y)\big )
\]
 Thus $p$ is convex in $x.$ By symmetry $p$ is convex in $y.$ The conclusion
 of the lemma now follows from Proposition \ref{p:HHLM}. 
\end{proof}

\subsection{The $xy$-Hessian}
\label{sec:xyHess}
 In view of Lemma \ref{l:xyconveximpliesbiconvex},  we now 
 consider only symmetric polynomials $p\in \PL.$
 Let 
$\{s_0,t_0,\alpha,\beta_j,\gamma,\delta_j: 0\le j \le 2\}$ denote freely 
noncommuting variables with $s_0,t_0,\beta_0,\beta_2,\delta_0,\delta_2$ 
symmetric. Let, in view of Proposition \ref{p:xy-convexp-alt},
\begin{equation*}
 s = \begin{pmatrix} s_0 & \begin{pmatrix} \alpha  & 0 \end{pmatrix} \\
   \begin{pmatrix} \alpha^* \\ 0 \end{pmatrix} & \begin{pmatrix} \beta_0 & \beta_1 \\
   \beta_1^* & \beta_2 \end{pmatrix} \end{pmatrix}, \quad
 t= \begin{pmatrix} t_0 & \begin{pmatrix} 0 & \gamma \end{pmatrix} \\
   \begin{pmatrix} 0\\ \gamma^*  \end{pmatrix} & \begin{pmatrix} \delta_0 & \delta_1 \\
   \delta_1^* & \delta_2 \end{pmatrix} \end{pmatrix},
   \quad V= \begin{pmatrix} 1 & 0 & 0\end{pmatrix}^*.
\end{equation*}
The \df{$xy$-Hessian}
 of   $p\in \C\langle x,y\rangle,$ denoted 
 \index{$\Hxy{p}$},   is the quadratic in $\alpha,\gamma$
 part of $V^* p(s,t)V-p(V^*(s,t)V) =V^* p(s,t)V - p(s_0,t_0).$ In
 particular, for  $p\in \PL,$
\[
  \Hxy{p}:= V^* p(s,t)V  -  p(V^*(s,t)V) =V^* p(s,t)V - p(s_0,t_0).
\]
The proof of the following lemma is routine. 

\begin{lemma}
\label{l:xyHessian}
 If  $p= \sum_{u\in \listL} p_u u \in \PL,$ then $\Hxy{p}$
 is a function of 
 $\{\alpha,\gamma,s_0,t_0,\delta_0,\delta_1, \beta_1,\beta_2\}$
with  the explicit form
\begin{equation*}
\begin{split}
\Hxy{p}&= [p_{x^2}\alpha \alpha^* +  p_{y^2}\gamma\gamma^*] + [p_{xyx} \alpha \delta_0\alpha^* + p_{yxy}\gamma \beta_2\gamma^*
  +  p_{xy^2}(s_0\gamma \gamma^* + \alpha \delta_1\gamma^*) \\
 &\phantom{=} + p_{y^2x}(\gamma \gamma^*s_0+\gamma \delta_1^* \alpha^*)
 +  p_{x^2y}(\alpha\alpha^*t_0+\alpha\beta_1 \gamma^*) + p_{yx^2} (t_0 \alpha\alpha^*+ \gamma\beta_1^*\alpha^*)]\\
&\phantom{=} + [p_{xy^2x}(s_0\gamma\gamma^*s_0+\alpha\delta_1\gamma^*s_0 + s_0 \gamma \delta_1^* \alpha^* + \alpha (\delta_0^2+\delta_1\delta_1^*) \alpha^*) \\
&\phantom{=} + p_{xyxy} (\alpha\delta_0\alpha^*t_0+\alpha\delta_0\beta_1\gamma^*+s_0\gamma\beta_2\gamma^*+\alpha\delta_1\beta_2\gamma^*)\\ 
&\phantom{=} +p_{yxyx} (t_0\alpha\delta_0\alpha^*+\gamma\beta_1^*\delta_0\alpha^*+\gamma\beta_2\gamma^*s_0+\gamma\beta_2\delta_1^*\alpha^*)\\
&\phantom{=} + p_{yx^2y} (t_0\alpha\alpha^*t_0+\gamma\beta_1^*\alpha^*t_0+t_0\alpha\beta_1\gamma^*+\gamma(\beta_1^*\beta_1+\beta_2^2)\gamma^*)] \\
&= \, \alpha \left [ p_{x^2} + p_{xyx}\delta_0 + p_{xy^2x}(\delta_0^2+\delta_1\delta_1^*)
       \right ]\alpha^*
  +\alpha \left [p_{xy^2} +p_{xyxy} \delta_0  \right ] \alpha^* t_0
  + t_0\alpha \left [p_{yx^2}+p_{yxyx}\delta_0 \right ] \alpha^* \\
  &\phantom{=} + \alpha \left [p_{xy^2}\delta_1 +p_{x^2y}\beta_1 +p_{xyxy} (\delta_0\beta_1
      + \delta_1\beta_2)     \right ] \gamma^* \\
  &\phantom{=}  +  \gamma \left [ p_{y^2x} \delta_1^* + p_{yx^2} \beta_1^* 
     +p_{yxyx}(\beta_1^*\delta_0 +\beta_2^* \delta_1)  \right ] \alpha^* \\
  &\phantom{=} +\alpha \left [ p_{xy^2x} \delta_1   \right]\gamma^* s_0 
  +  s_0 \gamma \left [ p_{xy^2x} \delta_1^* \right] \alpha^* 
  +  t_0 \alpha \left [p_{yx^2y}   \right] \alpha^* t_0 
  + t_0\alpha \left[p_{yx^2y} \beta_1  \right] \gamma^* \\
  &\phantom{=} +\gamma \left [ p_{yx^2y} \beta_1^* \right]\alpha^* t_0 
   + \gamma \left [ p_{y^2} + p_{yxy}\beta_2 + p_{yx^2y}(\beta_1^*\beta_1+\beta_2^2)
      \right ] \gamma^*
   + \gamma \left [p_{y^2x} +p_{yxyx}\beta_2   \right ] \gamma^* s_0 \\
 &\phantom{=} + s_0\gamma \left [p_{xy^2} + p_{xyxy} \beta_2 \right ] \gamma^* 
   +  s_0\gamma \left [ p_{yx^2y} \right ] \gamma^* s_0.
\end{split}
\end{equation*}
\end{lemma}

\begin{lemma}
\label{l:sameHxys}
 If $p\in \PL$ and $\Hxy{p}=0,$ then $p$ is an $xy$-pencil. 
 If $p,q\in \PL$ satisfy $\Hxy{p}=\Hxy{q},$ then there is an $xy$-pencil $\afflin\in \C\axy$ such 
 that $p=q+\afflin.$
\end{lemma}

\begin{proof}
 Since $\Hxy$ is  a linear mapping,
 it suffices to show, if $p=\sum_{w\in \listL} p_w w$
 satisfies  $\Hxy p=0$, then  $p$ is an $xy$-pencil.
 To this end, observe, if  $\Hxy{p}=0$, then, in view of
 Lemma \ref{l:xyHessian}, $p_w=0$ for  $w$ in the set
\[
  \{x^2,y^2,xyx,yxy,xy^2,y^2x,x^2y,yx^2,xy^2x,xyxy,yxyx,yx^2y\}.
\]
Hence the only possible nonzero coefficients of $p$ are $p_{1},p_x,p_y,p_{xy},p_{yx}$
and the  result follows.
\end{proof}

The Hessian of a $p\in \PL$  has a border vector-middle matrix representation that we
now describe. Since $p\in \PL,$
\[
p(x,y) = \afflin(x,y) + \sum_{w\in \listLp} p_w w,
\]
where $\afflin(x,y)$ is an $xy$-pencil and
\[
\listLp =\{x^2,y^2,xyx,yxy,xy^2,y^2x,x^2y,yx^2,xy^2x,xyxy,yxyx,yx^2y\}
 = \listL \setminus\{1,x,y,xy,yx\}.
\]
Since $p$ is symmetric, there are relations among its coefficients. For instance,
$p_{xyx},p_{yxy}\in\mathbb R$ and $p_{yx^2} =\overline{p_{x^2y}}$.

Let $\Bxy=\Bxy(s_0,t_0,\alpha,\gamma)$ denote the row vector-valued free polynomial, 
\[
 \Bxy(s_0,t_0,\alpha,\gamma) = \begin{pmatrix} \alpha &  t_0 \alpha   & 
   \gamma &  s_0 \gamma  \end{pmatrix}.
\]
We call \df{$\Bxy$} the  \df{$xy$-border vector}, or simply the border vector.

For $1\le j,k\le 2$,  let $\Mxy_{j,k}(\beta_1,\beta_2,\delta_0,\delta_1)$ denote the $2\times 2$
matrix polynomial,
\[
\begin{split}
 \Mxy_{11} &=  \begin{pmatrix} p_{x^2} + p_{xyx}\delta_0 + p_{xy^2x} (\delta_0^2 +\delta_1\delta_1^*)&
     p_{x^2y}+p_{xyxy}\delta_0 \\  p_{yx^2}  + p_{yxyx} \delta_0 & p_{yx^2y}\end{pmatrix},\\
 \Mxy_{12} &= \begin{pmatrix} p_{x^2y}\beta_1 +p_{xy^2} \delta_1 + p_{xyxy} (\delta_0\beta_1 + \delta_1\beta_2) &
   p_{xy^2x}\delta_1 \\ p_{yx^2y} \beta_1 & 0 \end{pmatrix}, \\
\Mxy_{21} &= \begin{pmatrix} p_{yx^2} \beta_1^* + p_{y^2x}\delta_1^* +p_{yxyx}(\beta_1^*\delta_0+\beta_2 \delta_1^*)
          & p_{yx^2y}\beta_1^*\\ p_{xy^2x}\delta_1^* & 0 \end{pmatrix}, \\
\Mxy_{22} &=\begin{pmatrix} p_{y^2}  + p_{yxy}\beta_2 + p_{yx^2y}(\beta_2^2 + \beta_1^* \beta_1)
  & p_{y^2x} +p_{yxyx}\beta_2 \\ p_{xy^2} + p_{xyxy} \beta_2 & p_{xy^2x} \end{pmatrix}.
\end{split}
\]
Let $\Mxy=(\Mxy_{j,k})_{j,k=1}^2$ denote the resulting $4\times 4$ ($2\times 2$ block matrix
with $2\times 2$ entries) matrix polynomial. 
The matrix \df{$\Mxy$} is the \df{$xy$-middle matrix}, or simply the
middle matrix, of  $p$.

\begin{lemma}
\label{l:xyBMB}
 If $p\in \PL$ is symmetric, then
\[
 \Hxy{p} = \Bxy \Mxy \Bxy^*.
\]
\end{lemma}

 Proposition~\ref{p:xyconvex-Mpos} shows 
$xy$-convexity of $p$ is equivalent to positivity of its middle matrix.

\begin{proposition}
\label{p:xyconvex-Mpos}
 If $p(x,y)$ is $xy$-convex, then $\Mxy(B_1,B_2,D_0,D_1)\succeq 0$
 for all matrices $(B_1,B_2,D_0,D_1)$ of compatible sizes. 
\end{proposition}

\begin{proof}
  Since $p$ is $xy$-convex, $\Hxy p\succeq 0.$ 
  Let positive integers $M,N$ and matrices  $D_0\in M_M(\C)$, $B_2\in M_N(\C)$ 
  and $B_1,D_1\in M_{N,M}(\C)$ be given. 
  Choose a vector $h\in\mathbb C^2$ and $X_0,Y_0\in \bS_2$ 
  such that $\{h,X_0h\}$ and $\{h,Y_0h\}$ are linearly independent.
  Positivity of the Hessian gives
\[
\begin{split}
0&\le  h^* \Hxy p(X_0,A,B_1,B_2,Y_0,C,D_0,D_1)h \\
  &  = [h^*\Bxy(X_0,A,Y_0,C)] \, \Mxy(B_1,B_2,D_0,D_1) \, [h^*\Bxy(X_0,A,Y_0,C)]^*.
\end{split}
\] 
  On the other hand, given vectors $f_1,\dots,f_4\in \C^M$, there exists 
 $A\in M_{2, M}(\C)$ and $C\in M_{2, N}(\C)$ 
 such that 
\[
 \Bxy(X_0,Y_0,A,C)^* h =   \begin{pmatrix} A^* h \\ A^* Y_0h \\ C^* h \\ C^* X_0 h\end{pmatrix}
   =\begin{pmatrix} f_1\\ f_2\\ f_3\\ f_4\end{pmatrix}. 
\]
It follows that $\Mxy(B_1,B_2,D_0,D_1) \succeq 0.$ 
\end{proof}

\subsection{Proof of Theorem \ref{t:introxyconvexp}}
\label{sec:proofxyconvexp}
The convexity assumption on $p$ implies the middle matrix $\Mxy$ of its Hessian 
takes positive semidefinite values by Proposition \ref{p:xyconvex-Mpos}. 

Let 
\[
\sigma=\left ( \begin{pmatrix} \delta_0 & \delta_1 \\ \delta_1^* & \delta_2 \end{pmatrix},
 \begin{pmatrix} \beta_0 & \beta_1 \\ \beta_1^* & \beta_2 \end{pmatrix} \right ).
\]
Let $Q$ denote the $2\times 2$ matrix polynomial obtained from
 the first and third rows and columns of $\Mxy$. Thus,
\begin{equation}
\label{def:Q}
 Q = Q(\delta_{a,b},\beta_{a,b})
\begin{pmatrix}
p_{x^2} + p_{xyx}\delta_0 + p_{xy^2x} (\delta_0^2 +\delta_1\delta_1^*) &
p_{x^2y}\beta_1 +p_{xy^2} \delta_1 + p_{xyxy} (\delta_0\beta_1 + \delta_1\beta_2) \\
p_{yx^2} \beta_1^* + p_{y^2x}\delta_1^* +p_{yxyx}(\beta_1^*\delta_0+\beta_2 \delta_1^*) &
p_{y^2}  + p_{yxy}\beta_2 + p_{yx^2y}(\beta_2^2 + \beta_1^* \beta_1)
\end{pmatrix},
\end{equation}
 and, given $S=(S_1,S_2)$ of the block form of equation~\eqref{eq:S2x2}, we have
 $Q(S)\succeq 0$ since $\Mxy(S_{2,1},S_{2,2},S_{1,0},S_{1,1})\succeq 0$
by Proposition~\ref{p:xyconvex-Mpos}.

 Define a $2\times 2$ polynomial $P(x_1,x_2) = \sum P_{j,k} x_j x_k$ (with $x_0=1$ as usual)
 by setting
\begin{equation}
\label{def:P}
\begin{split}
 P_{0,0} &=  \begin{pmatrix} p_{x^2}& 0\\ 0 &p_{y^2} \end{pmatrix}, \ \ 
 P_{0,1} = P_{1,0} =\frac 12 \begin{pmatrix} p_{xyx} & p_{xy^2} \\ p_{y^2x} & 0\end{pmatrix}, \ \
 P_{0,2}=P_{2,0} = \frac12 \begin{pmatrix} 0 & p_{x^2y} \\ 
       p_{yx^2} & p_{yxy}\end{pmatrix}\\
 P_{1,2} &= \begin{pmatrix} 0 & p_{xyxy} \\ 0 & 0 \end{pmatrix}, \ \
 P_{2,1}= \begin{pmatrix} 0& 0 \\ p_{yxyx} & 0 \end{pmatrix}, \ \
 P_{1,1} =  \begin{pmatrix} p_{xy^2x}&0\\0&0\end{pmatrix}, \ \ 
 P_{2,2} =  \begin{pmatrix} 0&0\\0& p_{yx^2y} \end{pmatrix}
\end{split}
\end{equation}
and observe   $\sE{P}(\sigma) = Q(\sigma).$ Thus $\sE{P}(S)\succeq 0$ for all 
tuples of hermitian matrices 
of the form \eqref{eq:S2x2}.
Hence Theorem \ref{l:FstarFisM} produces an $N$ and 
$F= \sum F_j s_j,$ where $F_j \in M_{N, 2}(\C),$
and an $R= \begin{pmatrix} 0 & r \\ r^* & 0\end{pmatrix}$  such that $F^*F +R= P,$
where $r\in \C$. In particular,
\[
\begin{split}
  F_j^* F_k &= P_{j,k}, \, 1\le j,k \le 2 \\
  F_0^*F_k + F_k^* F_0 &= P_{k,0}+P_{0,k}, \, k=1,2 \\
  F_0^* F_0 &= P_{0,0} +R,\\
  F_1^*F_1 = P_{1,1} =\begin{pmatrix} p_{xy^2x} &  0 \\ 0 & 0\end{pmatrix} &, \phantom{=} 
  F_2^*F_2 = P_{2,2} = \begin{pmatrix} 0&0\\0& p_{yx^2y} \end{pmatrix}.
  \end{split}
\]
Hence, letting  $\{\Ep_1,\Ep_2\}$ denote the standard orthonormal basis for $\C^2,$
$F_1e_2=0=F_2e_1$.  In particular, $e_1^* F_2^*F_0=0.$ Now 
 set $\Lambda_x = F_0e_1,$ $\Lambda_y=F_0e_2,$  $\Lambda_{yx} = F_1e_1$ and
$\Lambda_{xy} = F_2 e_2$ and verify,
\begin{align}\label{e:HxyqHxyp}
  \Lambda_x^* \Lambda_x &= e_1^*F_0^* F_0 e_1 = e_1^* P_{0,0} e_1 = p_{x^2} \notag \\
  \Lambda_y^* \Lambda_y &= e_2^*F_0^* F_0 e_2 = e_2^* P_{0,0} e_2 = p_{y^2} \notag  \\
  \Lambda_{yx}^* \Lambda_x +\Lambda_x^* \Lambda_{yx} 
   &= e_1^*F_1^*F_0e_1 + e_1^* F_0^*F_1 e_1
    = e_1^* (F_1^*F_0+F_0^*F_1)e_1 = (2P_{1,0})_{1,1} = p_{xyx} \notag \\
  \Lambda_{xy}^*\Lambda_y + \Lambda_y^*\Lambda_{xy} 
   &= e_2^*F_2^*F_0e_2 + e_2^*F_0^*F_2e_2 = e_2^*(F_2^*F_0 + F_0^*F_2)e_2
    = e_2^*(2P_{2,0})e_2 = p_{yxy}\notag \\
  \Lambda_x^*\Lambda_{xy} &= e_1^*F_0^* F_2 e_2 = e_1^* (F_0^*F_2 +F_2^*F_0)e_2
    = e_1^*(2P_{2,0})e_2 = p_{x^2 y}\notag \\
    \Lambda_y^*\Lambda_{yx} &= e_2^*F_0^*F_1e_1 = e_2^*(F_0^*F_1 +F_1^*F_0)e_1
     = e_2^*(2P_{1,0})e_1 =  p_{y^2x} \\
    \Lambda_{xy}^*\Lambda_{x} &= e_2^*F_2^*F_0e_1 = e_2^*(F_2^*F_0+F_0^*F_2)e_1
    =  e_2^*(2P_{2,0})e_1 =p_{yx^2} \notag \\
    \Lambda_{yx}^*\Lambda_{y} &=  e_1^*F_1^*F_0e_2 = e_1^*(F_1^*F_0+F_0F_1^*)e_2 
     = e_1^*(2P_{1,0})e_2 =p_{xy^2} \notag \\
  \Lambda_{yx}^*\Lambda_{yx} &= e_1^*F_1^*F_1e_1 = e_1^* P_{1,1} e_1 = p_{xy^2x}\notag \\
  \Lambda_{xy}^*\Lambda_{xy} &= e_2^*F_2^*F_2e_2 = e_2^* P_{2,2}e_2 =  p_{yx^2y}\notag  \\
  \Lambda_{xy}^*\Lambda_{yx} &= e_2^*F_2^*F_1e_1 = e_2^* P_{2,1} e_1 =p_{yxyx} \notag \\
  \Lambda_{yx}^* \Lambda_{xy} &= e_1^*F_1^* F_2 e_2 = e_1^* P_{1,2}e_2 = p_{xyxy}. \notag 
\end{align}

Let 
\[
 q=\Lambda(x,y,xy)^*\Lambda(x,y,xy),
\]
where $\Lambda$ denotes the $xy$-pencil 
\[
 \Lambda = \Lambda_{x}x+\Lambda_y y +\Lambda_{xy} xy +\Lambda_{yx} yx.
\]
A straightforward calculation, based on the identities of
equation \eqref{e:HxyqHxyp} and an appeal to the formula for the 
$xy$-Hessian in Lemma \ref{l:xyHessian}, shows $\Hxy q = \Hxy p$. Hence, by Lemma \ref{l:sameHxys},
there is a hermitian $xy$-pencil $\afflin$ 
such that $p=q+\afflin =\Lambda^*\Lambda + \afflin,$ completing the proof.
\qed

\begin{remark}\rm
 Note that  $\Lambda_x^* \Lambda_y +\Lambda_y^* \Lambda_x =R=\begin{pmatrix} 0 &r\\r^* & 0\end{pmatrix}.$
\end{remark}

\newpage

\appendix

\section{Not for publication}

\subsection{Proof of Proposition~\ref{p:a2convex}}
\label{s:a2convex}
 First suppose $r\in \fftwo$ is $a^2$-convex on 
 $\cD\subset \bS^{\tg}\times \bS^{\gv}.$ To prove $r$ is convex in $x$ 
 on $\cD,$ suppose $(A,X),(A,Y).$ 
  Consider the matrices 
\[ 
    B=\begin{pmatrix} A  & 0 \\ 0 & A \end{pmatrix}, \ \ 
    Z=\begin{pmatrix} X  & 0 \\ 0 & Y \end{pmatrix}, \ \ 
  V=\frac{1}{\sqrt{2}}\begin{pmatrix} I  \\ I \end{pmatrix}.
\] 
 Note that $V$ reduces $B.$  Equivalently  $V^* B^2V=(V^*BV)^2.$
 Since  $V^*(B,Z)V=(A,\frac{X+Y}{2})\in \cD$
 (by the convexity hypothesis on $\cD$) 
  and of course $(B,Z)\in \cD$ too,
\[
\frac{1}{2}(r(A,X)+r(A,Y))  =V^*r(B,Z)V \succeq  r(V^*(B,Z)V) = r\Big(A,\frac{X+Y}{2}\Big).
\]
Hence $r$ is convex in $x$ on $\cD.$

 Now suppose $r$ is convex in $x$ on $\cD$ and $(B,Z)\in \cD_n$,
 and  $V:\C^m\to \C^n$ is an isometry such that $V^*B^2V=(V^*BV)^2.$
 Thus the range of $V$ reduces $B$ and 
 up to unitary equivalence,
\[
  B= \begin{pmatrix} A& 0 \\ 0 & \alpha\end{pmatrix}, \ \
  Z=\begin{pmatrix} X & \beta \\ \beta^* & \delta \end{pmatrix}, \ \
  V=\begin{pmatrix} I \\ 0 \end{pmatrix}.
\]
 Let 
\[
 U =\begin{pmatrix} I & 0\\0 & -I \end{pmatrix}.
\]
Since $U$ is unitary, $\cD$ is a free set and  $(B,Z)\in \cD,$
 we have $U^*(B,Z)U\in \cD.$ 
Since $\cD$ is, by hypothesis, convex in $x,$
\[
\cD \ni \frac{(B,Z)+U^*(B,Z)U}{2} 
 = \left ( B, \begin{pmatrix} X&0\\0&\delta \end{pmatrix}\right ).
\]
Because $r$ is convex in $x,$
\[
 \begin{split}
\begin{pmatrix} r(V^*(B,Z)V) & 0 \\ 0 & r(\alpha,\beta) \end{pmatrix}
 & =  \begin{pmatrix} r(A,X) & 0 \\ 0 & r(\alpha,\beta)\end{pmatrix} 
 = r\left ( B, \begin{pmatrix} X&0\\0&\delta \end{pmatrix}\right )
\\ & \preceq \frac12 V^* \left (r(B,Z)+r(U^*(B,Z)U\right)  V
\\ &  = \frac12 V^* \left ( r(B,Z)+U^* r(B,Z)U\right ) V
 = V^* r(B,Z)V. 
\end{split}
\]
Thus $r$ is $a^2$-convex.\hfill\qedsymbol

\subsection{Proof of Lemma~\ref{l:xyHessian}}
 We provide the routine verification of the formula for $\Hxy{w}$ for  words $w\in\listL.$
 The result then follows  by linearity of $\Hxy.$

 It is clear that the $xy$-pencil terms ($1,x,y,xy$ and $yx$) vanish under $\Hxy.$
 The $(1,1)$ entry of $t^2$ is $t_0^2+\gamma\gamma^*.$ Thus, 
 \[
  H^{xy}{y^2} = t_0^2+\gamma\gamma^*-t_0^2 = \gamma\gamma^*.
 \]
 Similarly, the $(1,1)$ entry of $st^2$ is $s_0t_0^2 + s_0\gamma\gamma^*+\alpha\delta_1\gamma^*.$ 
 Hence,
 \[
 \Hxy{xy^2} = (s_0t_0^2 + s_0\gamma\gamma^*+\alpha\delta_1\gamma^*)- s_0t_0^2 =
   s_0\gamma\gamma^*+\alpha\delta_1\gamma^*.
 \]
  The $(1,1)$ entry of $sts$ is $(s_0t_0s_0+\alpha\delta_0\alpha^*) - s_0t_0s_0.$
 Hence,
\[
 \Hxy{xyx} = \alpha\delta_0\alpha^*.
\]
 The $(1,1)$ entry of $stst$ is 
$
  (s_0t_0s_0+\alpha \delta_0 \alpha^*)t_0 + (\alpha\delta_0\beta_1 
    +(s_0\gamma +\alpha \delta_1)\beta_2 )\gamma^*  - s_0t_0s_0t_0.$
 Hence
\[
 \Hxy{xyxy} = \alpha \delta_0 \alpha^* t_0 + \alpha (\delta_0 \beta_1
     + \delta_1\beta_2)  \gamma^*    +  s_0\gamma\beta_2 \gamma^*.
\]
 The $(1,1)$ entry of $st^2s$ is 
$
  s_0t_0^2s_0 + s_0\gamma\gamma^*s_0+\alpha\delta_1\gamma^*s_0 + s_0 \gamma \delta_1^* \alpha^* 
   + \alpha (\delta_0^2+\delta_1\delta_1^*) \alpha^*.$ 
 Thus, 
 \[
\begin{split}
 \Hxy{xy^2x} &= (s_0t_0^2s_0 + s_0\gamma\gamma^*s_0+\alpha\delta_1\gamma^*s_0 
     + s_0 \gamma \delta_1^* \alpha^* + \alpha (\delta_0^2+\delta_1\delta_1^*) \alpha^*)-s_0t_0^2s_0
\\  &=  s_0\gamma\gamma^*s_0+\alpha\delta_1\gamma^*s_0 + s_0 \gamma \delta_1^* \alpha^* 
    + \alpha (\delta_0^2+\delta_1\delta_1^*) \alpha^*.
\end{split}
  \] 
The remainder follow by symmetry in $x$ and $y$.\hfill\qedsymbol

\subsection{Examples}
\begin{example}\rm
\label{ex:square}
 Consider the polynomial
\[
\begin{split}
 p(x,y) &= x^2 +y^2 + xy^2x +2(xyxy+yxyx)+yx^2 y\\
 &=  \begin{pmatrix} 1 & y \end{pmatrix} \begin{pmatrix} x &0\\0 & x\end{pmatrix}
  \begin{pmatrix} 1+y^2  & 2y \\ 2y & 1 \end{pmatrix} \begin{pmatrix} x &0\\0 & x\end{pmatrix} 
 \begin{pmatrix} 1 \\ y \end{pmatrix}  + y^2\\
 &= \begin{pmatrix} 1 & x \end{pmatrix} \begin{pmatrix} y &0\\0 & y\end{pmatrix}
  \begin{pmatrix} 1+x^2  & 2x \\ 2x & 1 \end{pmatrix} \begin{pmatrix} y &0\\0 & y\end{pmatrix}
 \begin{pmatrix} 1 \\ x \end{pmatrix}+ x^2. 
\end{split}
\]
It is, by its very form, convex in $x$ on the free set $\{(X,Y): I-3 Y^2 \succeq 0\}$
and it is convex in $y$ on the set $\{(X,Y): I-3 X^2\succeq 0\}$. Thus $p$ is 
biconvex on $\mathscr D =\{(X,Y): I-3X^2,\, I-3Y^2 \succeq 0\}.$
That the set $\mathscr{E} =\{(X,Y):I-3X^2,\, I-3Y^2 \succ 0\}$ is
 the largest open free set on which $p$ is biconvex follows from
 Theorem \ref{t:wurzelification}.
\qed
\end{example}

\begin{example}\rm
\label{ex:returntosquare}
Consider the polynomial $p$ from Example \ref{ex:square}
and recall  $\mathscr D =\{(X,Y): I-3X^2,\, I-3Y^2 \succeq 0\}$
contains any free set on which $p$ is biconvex.  The middle matrix
 of the  $xy$-Hessian of $p$ is given by
\[
 \Mxy(x,y) = \begin{pmatrix} 
  \begin{pmatrix} I + \delta_0^2 +\delta_1\delta_1^* & 2\delta_0  \\ 2 \delta_0 & I \end{pmatrix}
  & \begin{pmatrix} 2(\delta_0 \beta_1 +\delta_1 \beta_2)  & \delta_1 \\
         \beta_1 & 0 \end{pmatrix} \\
   \begin{pmatrix} 2(\beta_1^*\delta_0 +\beta_2\delta_1^*) & \beta_1^* \\
                     \delta_1^* & 0 \end{pmatrix} 
  & \begin{pmatrix} I+\beta_2^2+\beta_1^*\beta_1 & 2\beta_2 \\ 2\beta_2 & I \end{pmatrix}
 \end{pmatrix}.
\]
Evidently $\Mxy\succeq 0$ in a neighborhood of $0$ and thus $p$ is $xy$-convex
in a neighborhood of $0$. On the other hand, $\Mxy$ is not positive semidefinite
on all of $\mathscr D$ and thus, arguing as in the proof of Theorem \ref{t:introxyconvexp},
$p$ is not $xy$-convex on the interior of $\mathscr{D}$.
\qed
\end{example}

\subsection{Equivalence of positivity of $\Mxy$ and $\sE{P}$}
\label{sec:curious}
 In this subsection we show directly  that positivity of $\Mxy$ 
 is equivalent to positivity of $\sE{P}=Q,$ where $P$ and $Q$ are
 defined in equations \eqref{def:P} and \eqref{def:Q}.
  Of course that positivity
 of $\Mxy$ implies positivity of $Q$ is immediate.  On the other
 hand, the proof of  Theorem \ref{t:introxyconvexp} only required
 the, a priori weaker, condition $\sE{P}\succeq 0$ and hence
 this latter condition implies positivity of the middle matrix
 $\Mxy.$ 

Assume 
 $Q$ takes positive semidefinite  values. Given  $\delta_1,\beta_1\in M_{n, m}(\C)$,
 $\delta_0\in M_n(\C)$ and $\beta_2\in M_m(\C)$, 
make the replacements,
\[
 \hatdelta_0 = \begin{pmatrix} \delta_0 & 0_{n\times m} \\ 
          0_{m\times n} &0_{m\times m}\end{pmatrix}, \ \
 \hatdelta_1 =\begin{pmatrix} \delta_1 & 0_{n\times n} \\
            tI_m  & 0_{m\times n} \end{pmatrix}, \ \
 \hatbeta_1 =\begin{pmatrix} \beta_1 & tI_n \\ 0_{m\times m} & 0_{m\times n} \end{pmatrix}, \ \
 \hatbeta_2 =\begin{pmatrix} \beta_2 &0_{m\times n}\\
  0_{n\times m}&0_{m\times m}\end{pmatrix}.
\]
 Substituing into $Q$ gives,
$
 Q(\hatbeta,\hatdelta) = \begin{pmatrix} Q_{j,k} \end{pmatrix}_{j,k=1}^2,$
 where
\[
\begin{split}
Q_{1,1}  &=
 \begin{pmatrix} p_{x^2} + p_{xyx} \delta_0 + p_{xy^2x} (\delta_0^2 +\delta_1\delta_1^*) &
   t p_{xy^2x} \delta_1 \\  tp_{xy^2x} \delta_1^*  & p_{x^2} + t^2 p_{xy^2x} \end{pmatrix} \\
Q_{1,2} &= 
 \begin{pmatrix} p_{x^2y}\beta_1 +p_{xy^2} \delta_1 + p_{xyxy} (\delta_0\beta_1 + \delta_1\beta_2)
    & t(p_{x^2 y} + p_{xyxy} \delta_0) \\ 
        t(p_{xy^2} + p_{xyxy} \beta_2) & 0 \end{pmatrix} = Q_{2,1}^*\\
Q_{2,2} &=
\begin{pmatrix} p_{y^2}  + p_{yxy}\beta_2 + p_{yx^2y}(\beta_2^2 + \beta_1^* \beta_1) &
  tp_{yx^2y} \beta_1^*  \\   tp_{yx^2y} \beta_1 & p_{y^2}+t^2 p_{yx^2y}
\end{pmatrix}.
\end{split}
\]
 Now conjugate each block with 
$
 \begin{pmatrix} 1 &0  \\ 0 &\frac{1}{t} \end{pmatrix}$
and let $t$ tend to infinity to deduce that 
$Q^\prime = \begin{pmatrix} Q^\prime_{j,k} \end{pmatrix}_{j,k=1}^2\succeq 0,$
 where
\[
\begin{split}
Q^\prime_{1,1}  &=
 \begin{pmatrix} p_{x^2} + p_{xyx} \delta_0 + p_{xy^2x} (\delta_0^2 +\delta_1\delta_1^*) &
    p_{xy^2x} \delta_1 \\  p_{xy^2x} \delta_1^*  &  p_{xy^2x} \end{pmatrix} \\
Q^\prime_{1,2} &= 
 \begin{pmatrix} p_{x^2y}\beta_1 +p_{xy^2} \delta_1 + p_{xyxy} (\delta_0\beta_1 + \delta_1\beta_2)
    & p_{x^2 y} + p_{xyxy} \delta_0 \\ 
           p_{xy^2} + p_{xyxy} \beta_2 & 0 \end{pmatrix} = (Q^\prime_{1,2})^*\\
Q^\prime_{2,2} &=
\begin{pmatrix} p_{y^2}  + p_{yxy}\beta_2 + p_{yx^2y}(\beta_2^2 + \beta_1^* \beta_1) &
  p_{yx^2y} \beta_1^*  \\   p_{yx^2y} \beta_1 & p_{yx^2y}
\end{pmatrix}.
\end{split}
\]
Finally, $Q^\prime$ is unitarily equivalent to $\Mxy$ via the permutation that
interchanges the second and fourth  rows and columns.

\end{document}